\documentclass[12pt,reqno]{amsart}
\usepackage[T1]{fontenc}
\usepackage[utf8]{inputenc}
\usepackage{lmodern}
\usepackage[a4paper,margin=1in]{geometry}
\usepackage{setspace}
\usepackage{amsmath,amssymb,amsthm,mathtools,amsfonts,bm}
\usepackage{esint}          
\numberwithin{equation}{section}
\usepackage{graphicx}
\usepackage{subcaption}
\usepackage{booktabs}
\usepackage{multirow}
\usepackage{tabularx}
\usepackage{colortbl}
\usepackage{float}
\usepackage{comment}
\usepackage{tikz}
\usetikzlibrary{arrows.meta,calc,positioning}
\usepackage{pgfplots}
\pgfplotsset{compat=1.18}
\usepackage{enumitem}
\usepackage{lscape}
\usepackage{hyperref}
\hypersetup{
  colorlinks=true,
  linkcolor=blue!50!black,
  citecolor=blue!50!black,
  urlcolor=blue!50!black,
  pdfauthor={Behrooz Moosavi Ramezanzadeh},
  pdftitle={A Projected Tug-of-War Game for the Regularized p-Laplacian}
}
\usepackage[nameinlink,capitalise,noabbrev]{cleveref}
\theoremstyle{plain}
\newtheorem{theorem}{Theorem}[section]
\newtheorem{proposition}[theorem]{Proposition}
\newtheorem{lemma}[theorem]{Lemma}
\newtheorem{corollary}[theorem]{Corollary}
\theoremstyle{definition}

\newtheorem{definition}[theorem]{Definition}

\theoremstyle{remark}
\newtheorem{remark}[theorem]{Remark}

\DeclareMathOperator{\divgg}{div}         
\DeclareMathOperator{\tr}{tr}            

\DeclareMathOperator{\dist}{dist}

\newcommand{\R}{\mathbb{R}}
\newcommand{\N}{\mathbb{N}}

\newcommand{\Geps}{\Gamma_\varepsilon}
\newcommand{\Oeps}{\Omega_\varepsilon}
\newcommand{\Be}{B_\varepsilon}
\newcommand{\Feps}{\mathcal{F}_\varepsilon}
\newcommand{\talpha}{\widetilde\alpha}
\newcommand{\tbeta}{\widetilde\beta}
\newcommand{\Sp}{S^+}
\newcommand{\Sm}{S^-}
\newcommand{\rhoeps}{\rho_\varepsilon}

\newcommand{\DeltapN}{\Delta_p^N}
\title[Projected Tug-of-War]{A Projected Tug-of-War Game for the Regularized 
$p$-Laplacian}
\author{Behrooz Moosavi Ramezanzadeh}
\thanks{University of Pittsburgh, Pittsburgh, PA, USA.
  \texttt{behroozmoosavi@pitt.edu}
  }
\date{\today}
\begin{document}
\begin{abstract}
We give a tug-of-war interpretation of the regularized $p$-Laplacian
$\divgg\big((1+|Dv|^2)^{p/2-1}Dv\big)=0$ in a bounded domain
$\Omega\subset\R^n$, $p\ge 2$.  The key is the linear lift
$w(x,x_{n+1})=v(x)+x_{n+1}$, which identifies this equation with
$\Delta_p w=0$ in $\R^{n+1}$.  Projecting the standard
$(n+1)$-dimensional $p$-harmonious scheme onto $\R^n$ yields a
discrete dynamic programming principle for which we prove existence,
uniqueness, and Borel measurability of solutions with strip boundary
data, identify the unique fixed point with the value of the projected
game, and establish convergence to the viscosity solution as
$\varepsilon\to 0$.
\end{abstract} 
\maketitle

\noindent\textbf{Acknowledgement. }I am deeply indebted to Professor Juan Manfredi for his proposed problem, guidance and valuable comments.

\section{Introduction}

The dynamic programming principle has become a central bridge between
nonlinear elliptic partial differential equations and stochastic games.
In the tug-of-war framework, the value function of a discrete
two-player zero-sum game satisfies a nonlinear dynamic programming
equation, and in the limit as the step size tends to zero this equation
converges to a continuum PDE.  This point of view has proved especially
effective for equations related to the $p$-Laplacian, where it provides
both a probabilistic interpretation and an elementary route to
existence, uniqueness, approximation, and qualitative properties of
solutions. \cite{PS08,MPR12,LPS14,CIL92}

For the normalized $p$-Laplacian, tug-of-war with noise leads to the
standard $p$-harmonious dynamic programming principle, and the
resulting theory is by now well understood. \cite{PS08,MPR12}
In particular, the strip-based formulation yields a natural framework
for fixed-point arguments, comparison principles, measurability, and
convergence to the corresponding viscosity solution. \cite{MPR12,LPS14}
Recent work has also emphasized the role of viscosity methods and
related superposition ideas in disjoint variables for nonlinear
equations, further clarifying the structural flexibility of this point
of view. \cite{CIL92,AMP25,LiuManfrediZhou25}

The present paper is concerned with the regularized $p$-Laplacian,
which for $p\ge 2$ may be written in divergence form as
\[
  \divgg\bigl((1+|Dv|^2)^{p/2-1}Dv\bigr)=0.
\]
Unlike the normalized $p$-Laplacian, this operator is uniformly
elliptic and does not degenerate at points where the gradient
vanishes.  It also arises naturally as the Euler--Lagrange equation of
an area-type functional.  The main problem studied here is whether this
regularized equation admits a natural tug-of-war interpretation, and
whether the corresponding dynamic programming principle can be treated
with the same precision as in the standard $p$-harmonious theory.

The key observation is that the regularized equation can be embedded
into the ordinary $p$-Laplacian in one higher dimension by the linear
lifting
\[
  w(x,x_{n+1})=v(x)+x_{n+1}.
\]
This is reminiscent of superposition constructions in disjoint
variables, although in the present setting the crucial fact is the
explicit identity obtained from the lifted gradient and Hessian, rather
than an abstract viscosity theorem on sums.  Projecting the standard
$(n+1)$-dimensional tug-of-war with noise back to $\R^n$ yields a new
dynamic programming principle: the averaging term becomes a weighted
average with a semicircular kernel, and the strategic sup/inf terms
acquire a geometric tilt determined by the lifted coordinate.

The purpose of the paper is to develop the basic theory of this
projected dynamic programming principle.  More precisely, we derive the
projected scheme from the higher-dimensional game, identify the
semicircular kernel and the tilted strategic terms, establish
existence, uniqueness, and measurability of solutions with strip
boundary data, and show that the unique solution coincides with the
value of the projected tug-of-war game.  The continuum limit is then
related back to the regularized $p$-Laplacian through the lifting and
the known convergence theory for $p$-harmonious functions in one higher
dimension.

The main contributions of the paper are as follows.

\begin{enumerate}[label=\textup{(\roman*)},leftmargin=3em]
  \item We identify the regularized $p$-Laplacian with the ordinary
  $p$-Laplacian of a lifted function in one higher dimension.

  \item We compute the projected noise law and obtain the semicircular
  kernel induced by the projection of the uniform measure on the
  $(n+1)$-dimensional ball.

  \item We derive the projected dynamic programming principle and
  identify its tilted strategic terms.

  \item We establish existence, uniqueness, and Borel measurability for
  the projected dynamic programming principle with strip boundary data.

  \item We prove that the unique fixed point of the projected operator
  is the value of the corresponding projected tug-of-war game.

  \item We formulate the convergence result linking the projected
  dynamic programming principle to the regularized $p$-Laplacian via
  the lifting construction.
\end{enumerate}

The paper is organized as follows.  Section~\ref{sec:lifting} derives
the lifting, the projected kernel, and the projected dynamic
programming principle.  Section~\ref{sec:operator} introduces the
function space and the projected dynamic programming operator.
Section~\ref{sec:existence} proves existence and uniqueness.
Section~\ref{sec:game} formulates the projected tug-of-war game and
identifies its value.  Section~\ref{sec:convergence} explains the
passage to the continuum limit through the higher-dimensional
$p$-harmonious theory.

Throughout the paper, $\Omega\subset\R^n$ is a bounded domain,
$\varepsilon>0$ denotes the step size, and $p\ge 2$ is fixed.
\section{The lifting and the projected DPP}\label{sec:lifting}

In this section we derive the discrete scheme that will be studied in
the remainder of the paper.  The central observation is that the
regularized $p$-Laplacian in $\R^n$ may be realized as the ordinary
$p$-Laplacian in one higher dimension through a linear lifting.
Projecting the standard $(n+1)$-dimensional tug-of-war with noise back
to $\R^n$ then produces the projected dynamic programming principle
introduced in the introduction.  The projection has two geometric
effects: the uniform noise in the higher-dimensional ball becomes a
weighted noise in $\R^n$, and the strategic sup/inf terms acquire a
tilt coming from the last coordinate.

\subsection{Heuristic identification of the game constants}
\label{subsec:constants}

We begin with the standard formal expansion that determines the
coefficients in the tug-of-war with noise dynamic programming
principle.  The computation is heuristic and is included only to
motivate the constants that will later arise from the lifting.

Let $u\in C^4(\Omega)$ and $x\in\Omega$.  Averaging the Taylor
expansion of $u$ over $B_\varepsilon(x)$ gives
\begin{equation}\label{eq:avg-expand}
  \fint_{B_\varepsilon(x)}u(y)\,dy
  =
  u(x)+\frac{\varepsilon^2}{2(n+2)}\Delta u(x)+O(\varepsilon^4),
\end{equation}
since the first moments vanish and
\[
  \fint_{B_\varepsilon(x)}(y_i-x_i)(y_j-x_j)\,dy
  =
  \frac{\varepsilon^2}{n+2}\,\delta_{ij}.
\]

If $u\in C^2(\Omega)$ and $Du(x)\neq 0$, then the maximizing and
minimizing directions over $B_\varepsilon(x)$ are asymptotically given
by $\pm Du(x)/|Du(x)|$.  Consequently,
\begin{equation}\label{eq:strat-expand}
  \frac12\Bigl(\sup_{B_\varepsilon(x)}u+\inf_{B_\varepsilon(x)}u\Bigr)
  =
  u(x)+\frac{\varepsilon^2}{2}\Delta_\infty^N u(x)+O(\varepsilon^3),
\end{equation}
where
\[
  \Delta_\infty^N u
  :=
  \frac{\langle D^2u\,Du,Du\rangle}{|Du|^2}
\]
is the normalized infinity Laplacian.

Suppose now that a dynamic programming principle in $\R^m$ has the form
\[
  u_\varepsilon(x)
  =
  \frac{\alpha}{2}\Bigl(\sup_{B_\varepsilon(x)}u_\varepsilon
        +\inf_{B_\varepsilon(x)}u_\varepsilon\Bigr)
  +\beta \fint_{B_\varepsilon(x)}u_\varepsilon(y)\,dy,
  \qquad \alpha+\beta=1.
\]
Substituting \eqref{eq:avg-expand} and \eqref{eq:strat-expand} and
subtracting $u(x)$ yields the formal consistency relation
\[
  0
  =
  \frac{\alpha\varepsilon^2}{2}\Delta_\infty^N u(x)
  +
  \frac{\beta\varepsilon^2}{2(m+2)}\Delta u(x)
  +O(\varepsilon^3).
\]
Thus, in order to match the normalized $p$-Laplace operator
\[
  \Delta u+(p-2)\Delta_\infty^N u,
\]
one must impose
\[
  (m+2)\alpha=(p-2)\beta,
  \qquad
  \alpha+\beta=1.
\]
Solving gives
\begin{equation}\label{eq:abm}
  \alpha=\frac{p-2}{p+m},
  \qquad
  \beta=\frac{m+2}{p+m}.
\end{equation}

In the present paper the relevant game lives, after lifting, in
dimension $m=n+1$.  Therefore the coefficients that appear in the
projected dynamic programming principle are
\begin{equation}\label{eq:constants}
  \talpha=\frac{p-2}{p+n+1},
  \qquad
  \tbeta=\frac{n+3}{p+n+1},
  \qquad
  \talpha+\tbeta=1.
\end{equation}
These are exactly the standard $p$-harmonious coefficients in
dimension $n+1$. \cite{MPR12,LPS14}

\begin{remark}
The preceding computation is only a consistency check for the choice of
coefficients.  In the regularized problem the true mechanism is not the
formal comparison with the normalized $p$-Laplacian in $\R^n$, but the
lifting to the ordinary $p$-Laplacian in $\R^{n+1}$ established below.
\end{remark}

\subsection{The lifting equivalence}
\label{subsec:lift}

We now prove the basic structural identity that underlies the whole
paper.

\begin{proposition}[Lifting equivalence]
\label{prop:lift}
Let $\Omega\subset\R^n$ be open and let $p\ge 2$.  For
$v\in C^2(\Omega)$ define
\[
  w:\Omega\times\R\to\R,
  \qquad
  w(x,x_{n+1})=v(x)+x_{n+1}.
\]
Then
\[
  \divgg\!\bigl((1+|Dv|^2)^{p/2-1}Dv\bigr)=0
  \quad\text{in }\Omega
\]
if and only if
\[
  \Delta_p w=\divgg\!\bigl(|Dw|^{p-2}Dw\bigr)=0
  \quad\text{in }\Omega\times\R.
\]
\end{proposition}

\begin{proof}
Write points in $\R^{n+1}$ as $(x',x_{n+1})$, with $x'\in\R^n$.
Since
\[
  w(x',x_{n+1})=v(x')+x_{n+1},
\]
we have
\[
  Dw=(Dv,1),
  \qquad
  |Dw|^2=|Dv|^2+1,
\]
and
\[
  D^2w
  =
  \begin{pmatrix}
    D^2v & 0\\
    0 & 0
  \end{pmatrix}.
\]
In particular, $Dw\neq 0$ everywhere because the last component equals
$1$.

Now
\[
  \Delta_p w
  =
  \sum_{k=1}^{n+1}\partial_{x_k}\!\bigl(|Dw|^{p-2}\partial_{x_k}w\bigr).
\]
Since $|Dw|^2=1+|Dv(x')|^2$ is independent of $x_{n+1}$ and
$\partial_{x_{n+1}}w=1$, the last term vanishes:
\[
  \partial_{x_{n+1}}\!\bigl(|Dw|^{p-2}\partial_{x_{n+1}}w\bigr)
  =
  \partial_{x_{n+1}}\!\bigl((1+|Dv|^2)^{p/2-1}\bigr)
  =0.
\]
Therefore
\begin{align*}
  \Delta_p w
  &=
  \sum_{i=1}^n
  \partial_{x_i}\!\bigl((1+|Dv|^2)^{p/2-1}\partial_{x_i}v\bigr)\\
  &=
  \divgg\!\bigl((1+|Dv|^2)^{p/2-1}Dv\bigr).
\end{align*}
Thus $\Delta_p w=0$ in $\Omega\times\R$ if and only if
\[
  \divgg\!\bigl((1+|Dv|^2)^{p/2-1}Dv\bigr)=0
\]
in $\Omega$.
\end{proof}

\begin{remark}
The regularization in the operator is encoded geometrically by the
extra linear coordinate.  Indeed, the identity $|Dw|^2=1+|Dv|^2$
explains precisely why the standard $p$-Laplacian of $w$ becomes the
regularized $p$-Laplacian of $v$.
\end{remark}

\subsection{The projected kernel}
\label{subsec:kernel}

The noise term in the lifted game is the uniform average over the
$(n+1)$-dimensional ball.  Projecting this average onto $\R^n$ produces
a weighted kernel, which we now compute.

\begin{lemma}[Semicircular kernel]
\label{lem:density}
Define
\begin{equation}\label{eq:rho}
  \rhoeps(h)
  :=
  \frac{2\sqrt{\varepsilon^2-|h|^2}}{|B_\varepsilon^{n+1}|},
  \qquad h\in B_\varepsilon(0)\subset\R^n.
\end{equation}
Then $\rhoeps$ is a strictly positive probability density on
$B_\varepsilon(0)$.
\end{lemma}

\begin{proof}
If $|h|<\varepsilon$, then $\varepsilon^2-|h|^2>0$, so
$\rhoeps(h)>0$.

It remains to check the normalization.  We claim that
\[
  \int_{B_\varepsilon(0)}2\sqrt{\varepsilon^2-|h|^2}\,dh
  =
  |B_\varepsilon^{n+1}|.
\]
Indeed, for each fixed $h\in\R^n$ with $|h|<\varepsilon$, the fiber of
the $(n+1)$-dimensional ball above $h$ is the interval
\[
  \Bigl\{t\in\R: |h|^2+t^2<\varepsilon^2\Bigr\}
  =
  \Bigl(-\sqrt{\varepsilon^2-|h|^2},\sqrt{\varepsilon^2-|h|^2}\Bigr),
\]
whose length is exactly $2\sqrt{\varepsilon^2-|h|^2}$.  Integrating the
fiber length over $B_\varepsilon(0)\subset\R^n$ gives the measure of the
whole ball $B_\varepsilon^{n+1}(0)\subset\R^{n+1}$:
\[
  |B_\varepsilon^{n+1}|
  =
  \int_{B_\varepsilon(0)}2\sqrt{\varepsilon^2-|h|^2}\,dh.
\]
Therefore
\[
  \int_{B_\varepsilon(0)}\rhoeps(h)\,dh
  =
  \frac{1}{|B_\varepsilon^{n+1}|}
  \int_{B_\varepsilon(0)}2\sqrt{\varepsilon^2-|h|^2}\,dh
  =
  1.
\]
\end{proof}

\begin{remark}
The kernel $\rhoeps$ is radial and strictly positive on the whole open
ball $B_\varepsilon(0)$.  It is the pushforward of the normalized
Lebesgue measure on $B_\varepsilon^{n+1}(0)$ under the projection
$(h,t)\mapsto h$.
\end{remark}

\subsection{Derivation of the projected dynamic programming principle}
\label{subsec:project}

We now derive the projected dynamic programming principle from the
standard tug-of-war with noise in one higher dimension.

Let
\[
  \Geps
  :=
  \{x\in\R^n\setminus\Omega:\dist(x,\Omega)\le\varepsilon\},
  \qquad
  \Oeps:=\Omega\cup\Geps.
\]
Assume that $v_\varepsilon:\Oeps\to\R$ is extended to the strip by the
prescribed boundary data.  Define the lifted function
\[
  w_\varepsilon:\Oeps\times\R\to\R,
  \qquad
  w_\varepsilon(x,s):=v_\varepsilon(x)+s.
\]
Suppose that $w_\varepsilon$ satisfies the standard
$(n+1)$-dimensional DPP
\begin{equation}\label{eq:DPP-n+1}
  w_\varepsilon(x,s)
  =
  \frac{\talpha}{2}
  \left(
    \sup_{B_\varepsilon^{n+1}(x,s)}w_\varepsilon
    +
    \inf_{B_\varepsilon^{n+1}(x,s)}w_\varepsilon
  \right)
  +
  \tbeta\fint_{B_\varepsilon^{n+1}(x,s)}w_\varepsilon(\xi)\,d\xi.
\end{equation}
We show that $v_\varepsilon$ then satisfies a projected
$n$-dimensional dynamic programming principle.

For $u:\Oeps\to\R$ and $x\in\Omega$, define the tilted functionals
\begin{align}
  \Sp[u](x)
  &:=
  \sup_{\substack{\tilde x\in\Oeps\\|\tilde x-x|\le\varepsilon}}
  \Bigl(u(\tilde x)+\sqrt{\varepsilon^2-|\tilde x-x|^2}\Bigr),
  \label{eq:Splus}\\
  \Sm[u](x)
  &:=
  \inf_{\substack{\tilde x\in\Oeps\\|\tilde x-x|\le\varepsilon}}
  \Bigl(u(\tilde x)-\sqrt{\varepsilon^2-|\tilde x-x|^2}\Bigr).
  \label{eq:Sminus}
\end{align}

\begin{proposition}[Projected dynamic programming principle]
\label{prop:dpp-intro}
If $w_\varepsilon(x,s)=v_\varepsilon(x)+s$ satisfies
\eqref{eq:DPP-n+1}, then $v_\varepsilon$ satisfies
\begin{equation}\label{eq:DPP-v}
  v_\varepsilon(x)
  =
  \frac{\talpha}{2}\bigl(\Sp[v_\varepsilon](x)+\Sm[v_\varepsilon](x)\bigr)
  +
  \tbeta\fint_{B_\varepsilon(0)}v_\varepsilon(x+h)\rhoeps(h)\,dh,
  \qquad x\in\Omega.
\end{equation}
\end{proposition}

\begin{proof}
Fix $(x,s)\in\Omega\times\R$.

For the strategic terms, observe that
$(\tilde x,\tilde s)\in B_\varepsilon^{n+1}(x,s)$ if and only if
\[
  |\tilde x-x|^2+|\tilde s-s|^2<\varepsilon^2.
\]
Thus, for each admissible $\tilde x$, the variable $\tilde s$ ranges over
\[
  s-\sqrt{\varepsilon^2-|\tilde x-x|^2}
  <
  \tilde s
  <
  s+\sqrt{\varepsilon^2-|\tilde x-x|^2}.
\]
Since $w_\varepsilon(\tilde x,\tilde s)=v_\varepsilon(\tilde x)+\tilde s$,
maximization in $\tilde s$ gives
\[
  \sup_{B_\varepsilon^{n+1}(x,s)}w_\varepsilon
  =
  s+\Sp[v_\varepsilon](x),
\]
and minimization gives
\[
  \inf_{B_\varepsilon^{n+1}(x,s)}w_\varepsilon
  =
  s+\Sm[v_\varepsilon](x).
\]

For the averaging term, write points in $\R^{n+1}$ as $(x+h,s+t)$,
with $(h,t)\in B_\varepsilon^{n+1}(0)$.  Then
\[
  w_\varepsilon(x+h,s+t)=v_\varepsilon(x+h)+s+t.
\]
Therefore
\begin{align*}
  \fint_{B_\varepsilon^{n+1}(x,s)}w_\varepsilon(\xi)\,d\xi
  &=
  \frac1{|B_\varepsilon^{n+1}|}
  \int_{B_\varepsilon^{n+1}(0)}
  \bigl(v_\varepsilon(x+h)+s+t\bigr)\,dh\,dt\\
  &=
  s+
  \frac1{|B_\varepsilon^{n+1}|}
  \int_{B_\varepsilon(0)}
  \left(
    \int_{-\sqrt{\varepsilon^2-|h|^2}}^{\sqrt{\varepsilon^2-|h|^2}}
    \bigl(v_\varepsilon(x+h)+t\bigr)\,dt
  \right)dh.
\end{align*}
The odd part in $t$ vanishes, and the remaining integral becomes
\begin{align*}
  \fint_{B_\varepsilon^{n+1}(x,s)}w_\varepsilon(\xi)\,d\xi
  &=
  s+
  \frac1{|B_\varepsilon^{n+1}|}
  \int_{B_\varepsilon(0)}
  2\sqrt{\varepsilon^2-|h|^2}\,v_\varepsilon(x+h)\,dh\\
  &=
  s+\int_{B_\varepsilon(0)}v_\varepsilon(x+h)\rhoeps(h)\,dh.
\end{align*}

Substituting these identities into \eqref{eq:DPP-n+1} gives
\begin{align*}
  v_\varepsilon(x)+s
  &=
  \frac{\talpha}{2}
  \Bigl(s+\Sp[v_\varepsilon](x)+s+\Sm[v_\varepsilon](x)\Bigr)\\
  &\qquad
  +\tbeta
  \left(
    s+\int_{B_\varepsilon(0)}v_\varepsilon(x+h)\rhoeps(h)\,dh
  \right).
\end{align*}
Since $\talpha+\tbeta=1$, the terms involving $s$ cancel, and we obtain
\eqref{eq:DPP-v}.
\end{proof}

\begin{remark}
The projected DPP differs from the standard $p$-harmonious equation in
two essential ways.  The noise is no longer uniform, but governed by the
semicircular kernel $\rhoeps$, and the strategic terms are no longer plain
sup/inf, but the tilted operators \eqref{eq:Splus}--\eqref{eq:Sminus}.
These are precisely the two traces left in $\R^n$ by the higher-dimensional
geometry of the lifted game.
\end{remark}
\section{The function space and the projected DPP operator}
\label{sec:operator}

In this section we formulate the projected dynamic programming
principle on the natural strip domain and establish the basic analytic
properties of the associated operator in the sprit of \cite{MPR12}.  The main point is that the
projected scheme is well defined on the strip extension $\Oeps$, and
that the tilted strategic terms preserve Borel measurability.

\subsection{The strip domain and the function space}

Let $\Omega\subset\R^n$ be a bounded domain and let $\varepsilon>0$.
As in the standard tug-of-war with noise theory, boundary data must be
prescribed on the $\varepsilon$-strip
\begin{equation}\label{eq:strip-domain}
  \Geps
  :=
  \{x\in\R^n\setminus\Omega:\dist(x,\Omega)\le \varepsilon\},
  \qquad
  \Oeps:=\Omega\cup\Geps.
\end{equation}
Indeed, if $x\in\Omega$ and $|h|<\varepsilon$, then $x+h$ need not lie
in $\Omega$, but it always belongs to $\Oeps$.  Likewise, any admissible
strategic move $\tilde x$ with $|\tilde x-x|\le \varepsilon$ lies in
$\Oeps$ once the process is formulated on the strip.

Let $F:\Geps\to\R$ be a bounded Borel function.  We define
\begin{equation}\label{eq:Feps}
  \Feps
  :=
  \bigl\{
    u:\Oeps\to\R:
    u \text{ is bounded and Borel, and } u|_{\Geps}=F
  \bigr\}.
\end{equation}
We equip $\Feps$ with the supremum norm
\[
  \|u\|_\infty:=\sup_{x\in\Oeps}|u(x)|.
\]

\begin{remark}
The use of strip boundary data is not merely technical.  Since the
projected DPP contains both the values $u(x+h)$ for $|h|<\varepsilon$
and the tilted strategic terms evaluated at points
$\tilde x$ with $|\tilde x-x|\le \varepsilon$, the natural state space
of the scheme is $\Oeps$, not just $\overline\Omega$.
\end{remark}

\subsection{The tilted strategic functionals}

The projection of the lifted game replaces the usual supremum and
infimum over $B_\varepsilon(x)$ by tilted versions.

\begin{definition}[Tilted sup/inf]
\label{def:tilted}
For a bounded function $u:\Oeps\to\R$ and $x\in\Omega$, define
\begin{align}
  \Sp[u](x)
  &:=
  \sup_{\substack{\tilde x\in\Oeps\\ |\tilde x-x|\le \varepsilon}}
  \Bigl(u(\tilde x)+\sqrt{\varepsilon^2-|\tilde x-x|^2}\Bigr),
  \label{eq:Splus-def}\\
  \Sm[u](x)
  &:=
  \inf_{\substack{\tilde x\in\Oeps\\ |\tilde x-x|\le \varepsilon}}
  \Bigl(u(\tilde x)-\sqrt{\varepsilon^2-|\tilde x-x|^2}\Bigr).
  \label{eq:Sminus-def}
\end{align}
\end{definition}

The first issue is measurability.

\begin{lemma}[Borel measurability of the tilted functionals]
\label{lem:tilt-meas}
If $u:\Oeps\to\R$ is bounded and Borel, then $\Sp[u]$ and $\Sm[u]$ are
Borel measurable on $\Omega$.
\end{lemma}

\begin{proof}
We prove the statement for $\Sp[u]$; the argument for $\Sm[u]$ is
analogous.

Fix $\lambda\in\R$.  We show that the superlevel set
\[
  E_\lambda:=\{x\in\Omega:\Sp[u](x)>\lambda\}
\]
is open in $\Omega$.

By definition, $x\in E_\lambda$ if and only if there exists
$\tilde x\in\Oeps$ with $|\tilde x-x|\le \varepsilon$ such that
\[
  u(\tilde x)+\sqrt{\varepsilon^2-|\tilde x-x|^2}>\lambda.
\]
We analyze this condition according to the size of $u(\tilde x)$.

If $u(\tilde x)>\lambda$, then the above inequality is automatically
satisfied for every $x\in B_\varepsilon(\tilde x)\cap\Omega$, since the
square-root term is nonnegative.  Hence such a point $\tilde x$
contributes the whole ball $B_\varepsilon(\tilde x)\cap\Omega$.

If $\lambda-\varepsilon<u(\tilde x)\le \lambda$, then
$\lambda-u(\tilde x)\in[0,\varepsilon)$ and the condition is equivalent
to
\[
  \sqrt{\varepsilon^2-|\tilde x-x|^2}>\lambda-u(\tilde x),
\]
which in turn is equivalent to
\[
  |\tilde x-x|^2<\varepsilon^2-(\lambda-u(\tilde x))^2.
\]
Thus such a point $\tilde x$ contributes the open ball
\[
  B_{\mu(\tilde x)}(\tilde x)\cap\Omega,
  \qquad
  \mu(\tilde x):=
  \sqrt{\varepsilon^2-(\lambda-u(\tilde x))^2}.
\]

Finally, if $u(\tilde x)\le \lambda-\varepsilon$, then
\[
  u(\tilde x)+\sqrt{\varepsilon^2-|\tilde x-x|^2}
  \le u(\tilde x)+\varepsilon
  \le \lambda,
\]
so such $\tilde x$ contributes nothing.

Therefore
\[
  E_\lambda
  =
  \Omega\cap
  \bigcup_{\substack{\tilde x\in\Oeps\\ u(\tilde x)>\lambda-\varepsilon}}
  B_{\mu(\tilde x)}(\tilde x),
\]
where $\mu(\tilde x)=\varepsilon$ when $u(\tilde x)>\lambda$ and
\[
  \mu(\tilde x)=\sqrt{\varepsilon^2-(\lambda-u(\tilde x))^2}
\]
when $\lambda-\varepsilon<u(\tilde x)\le \lambda$.
This is a union of open sets intersected with $\Omega$, hence is open.
Thus $\Sp[u]$ is Borel measurable.

The proof for $\Sm[u]$ is the same after replacing $u$ by $-u$.
\end{proof}

\begin{remark}
The above argument is slightly different from the standard one for
\[
  x\mapsto \sup_{B_\varepsilon(x)}u,
\]
because the tilt term depends on the distance from $x$ to the strategic
point $\tilde x$.  The case distinction above is what makes the
measurability argument work in the present setting.
\end{remark}

\subsection{The projected dynamic programming operator}

We now define the operator associated with the projected dynamic
programming principle.

\begin{definition}[Projected DPP operator]
\label{def:T}
For $u\in\Feps$, define $Tu:\Oeps\to\R$ by
\begin{equation}\label{eq:T-boundary}
  Tu(x):=F(x),
  \qquad x\in\Geps,
\end{equation}
and
\begin{equation}\label{eq:T-interior}
  Tu(x)
  :=
  \frac{\talpha}{2}\bigl(\Sp[u](x)+\Sm[u](x)\bigr)
  +\tbeta\int_{B_\varepsilon(0)}u(x+h)\rhoeps(h)\,dh,
  \qquad x\in\Omega,
\end{equation}
where the coefficients $\talpha,\tbeta$ are given by
\eqref{eq:constants} and the kernel $\rhoeps$ is given by
\eqref{eq:rho}.
\end{definition}

A function $u\in\Feps$ is a solution of the projected dynamic
programming principle if and only if
\[
  Tu=u
  \qquad\text{in }\Oeps.
\]

We next collect the basic properties of $T$.

\begin{lemma}[Basic properties of the projected DPP operator]
\label{lem:T-props}
The operator $T:\Feps\to\Feps$ has the following properties.

\begin{enumerate}[label=\textup{(\roman*)},leftmargin=2.5em]
  \item \textbf{Well definedness.}  If $u\in\Feps$, then $Tu\in\Feps$.

  \item \textbf{Monotonicity.}  If $u,v\in\Feps$ and $u\le v$ in
  $\Oeps$, then $Tu\le Tv$ in $\Oeps$.

  \item \textbf{Supremum norm stability.}  If $u\in\Feps$, then
  \[
    \|Tu\|_\infty
    \le \max\{\|F\|_{L^\infty(\Geps)},\,\|u\|_\infty+\talpha\varepsilon\}.
  \]

  \item \textbf{Nonexpansiveness.}  If $u,v\in\Feps$, then
  \[
    \|Tu-Tv\|_\infty\le \|u-v\|_\infty.
  \]
\end{enumerate}
\end{lemma}

\begin{proof}
\textit{(i) Well definedness.}
Let $u\in\Feps$.  On $\Geps$, $Tu=F$ by definition, so only the interior
part requires verification.

By Lemma~\ref{lem:tilt-meas}, both $\Sp[u]$ and $\Sm[u]$ are Borel on
$\Omega$.  Since $u$ is bounded and Borel on $\Oeps$, the map
\[
  (x,h)\mapsto u(x+h)\rhoeps(h)
\]
is Borel on $\Omega\times B_\varepsilon(0)$ and bounded by
$\|u\|_\infty\rhoeps(h)$.  Therefore
\[
  x\mapsto \int_{B_\varepsilon(0)}u(x+h)\rhoeps(h)\,dh
\]
is Borel measurable by Fubini's theorem.  Hence $Tu$ is Borel on
$\Omega$, and therefore on $\Oeps$.

It remains to check boundedness.  If $|u|\le M$ on $\Oeps$, then for
$x\in\Omega$,
\[
  \Sp[u](x)\le M+\varepsilon,
  \qquad
  \Sm[u](x)\ge -M-\varepsilon,
\]
and similarly
\[
  \Sp[u](x)\ge -M,
  \qquad
  \Sm[u](x)\le M.
\]
A crude bound is therefore
\[
  |\Sp[u](x)|\le M+\varepsilon,
  \qquad
  |\Sm[u](x)|\le M+\varepsilon.
\]
Moreover, since $\rhoeps$ is a probability density,
\[
  \left|
    \int_{B_\varepsilon(0)}u(x+h)\rhoeps(h)\,dh
  \right|
  \le M.
\]
Thus
\[
  |Tu(x)|
  \le
  \frac{\talpha}{2}(M+\varepsilon+M+\varepsilon)+\tbeta M
  =
  M+\talpha\varepsilon.
\]
Hence $Tu$ is bounded and belongs to $\Feps$.

\textit{(ii) Monotonicity.}
Assume $u\le v$ in $\Oeps$.  Then for every admissible $\tilde x$,
\[
  u(\tilde x)+\sqrt{\varepsilon^2-|\tilde x-x|^2}
  \le
  v(\tilde x)+\sqrt{\varepsilon^2-|\tilde x-x|^2},
\]
and similarly
\[
  u(\tilde x)-\sqrt{\varepsilon^2-|\tilde x-x|^2}
  \le
  v(\tilde x)-\sqrt{\varepsilon^2-|\tilde x-x|^2}.
\]
Taking suprema and infima gives
\[
  \Sp[u](x)\le \Sp[v](x),
  \qquad
  \Sm[u](x)\le \Sm[v](x).
\]
Since $\rhoeps\ge 0$, we also have
\[
  \int_{B_\varepsilon(0)}u(x+h)\rhoeps(h)\,dh
  \le
  \int_{B_\varepsilon(0)}v(x+h)\rhoeps(h)\,dh.
\]
Therefore $Tu\le Tv$ on $\Omega$, and the boundary values are identical
on $\Geps$.

\textit{(iii) Supremum norm stability.}
This was already proved in part \textit{(i)} on $\Omega$, while on
$\Geps$ we have $Tu=F$.  Hence
\[
  \|Tu\|_\infty
  \le
  \max\{\|F\|_{L^\infty(\Geps)},\,\|u\|_\infty+\talpha\varepsilon\}.
\]

\textit{(iv) Nonexpansiveness.}
Let $\delta:=\|u-v\|_\infty$.  Then for every admissible $\tilde x$,
\[
  u(\tilde x)\le v(\tilde x)+\delta,
  \qquad
  v(\tilde x)\le u(\tilde x)+\delta.
\]
Adding the same tilt term and taking suprema yields
\[
  |\Sp[u](x)-\Sp[v](x)|\le \delta.
\]
The same argument gives
\[
  |\Sm[u](x)-\Sm[v](x)|\le \delta.
\]
For the integral term,
\[
  \left|
    \int_{B_\varepsilon(0)}(u-v)(x+h)\rhoeps(h)\,dh
  \right|
  \le
  \delta\int_{B_\varepsilon(0)}\rhoeps(h)\,dh
  =\delta.
\]
Therefore, for $x\in\Omega$,
\begin{align*}
  |Tu(x)-Tv(x)|
  &\le
  \frac{\talpha}{2}|\Sp[u](x)-\Sp[v](x)|
  +
  \frac{\talpha}{2}|\Sm[u](x)-\Sm[v](x)|\\
  &\qquad
  +
  \tbeta
  \left|
    \int_{B_\varepsilon(0)}(u-v)(x+h)\rhoeps(h)\,dh
  \right|\\
  &\le
  \frac{\talpha}{2}\delta+\frac{\talpha}{2}\delta+\tbeta\delta
  =
  (\talpha+\tbeta)\delta
  =
  \delta.
\end{align*}
On $\Geps$, both $Tu$ and $Tv$ equal $F$.  Thus
\[
  \|Tu-Tv\|_\infty\le \delta=\|u-v\|_\infty.
\]
\end{proof}

\begin{remark}
The operator $T$ is monotone and nonexpansive, but in general it is not
a strict contraction in the supremum norm.  The existence theory will
therefore rely on monotone iteration rather than the Banach fixed point
theorem.
\end{remark}

\subsection{Fixed points and the projected DPP}

We conclude this section by recording the fixed-point formulation that
will be used throughout the paper.

\begin{definition}[Solution of the projected DPP]
\label{def:solution-dpp}
A function $v_\varepsilon\in\Feps$ is called a solution of the
projected dynamic programming principle if
\[
  Tv_\varepsilon=v_\varepsilon
  \qquad\text{in }\Oeps.
\]
Equivalently,
\[
  v_\varepsilon(x)
  =
  \frac{\talpha}{2}\bigl(\Sp[v_\varepsilon](x)+\Sm[v_\varepsilon](x)\bigr)
  +
  \tbeta\int_{B_\varepsilon(0)}v_\varepsilon(x+h)\rhoeps(h)\,dh
\]
for every $x\in\Omega$, together with the boundary condition
$v_\varepsilon=F$ on $\Geps$.
\end{definition}
In other words,
\begin{equation}\label{eq:DPP-v-boundary}
\begin{cases}
v_\varepsilon(x)
=
\dfrac{\talpha}{2}\bigl(\Sp[v_\varepsilon](x)+\Sm[v_\varepsilon](x)\bigr)
+\tbeta\displaystyle\int_{B_\varepsilon(0)}v_\varepsilon(x+h)\rhoeps(h)\,dh,
& x\in\Omega,\\[1.2ex]
v_\varepsilon(x)=F(x), & x\in\Geps.
\end{cases}
\end{equation}

The existence, uniqueness, and game-theoretic interpretation of such
fixed points will be established in the following sections.

\section{Existence and uniqueness}
\label{sec:existence}

In this section we prove that the projected dynamic programming
principle has a unique bounded Borel solution with prescribed strip
boundary data.  The proof follows the monotone-iteration strategy used
for the standard $p$-harmonious functions, but it must be adapted to
the present setting because the uniform averaging operator is replaced
by the semicircular kernel and the ordinary sup/inf terms are replaced
by the tilted functionals.

\subsection{Existence by monotone iteration}

We begin with the existence theorem.

\begin{theorem}[Existence]
\label{thm:exist}
Let $\Omega\subset\R^n$ be bounded, let $\varepsilon>0$, and let
$F:\Geps\to\R$ be a bounded Borel function.  Then there exists a
bounded Borel function $v_\varepsilon\in\Feps$ such that
\[
  Tv_\varepsilon=v_\varepsilon
  \qquad\text{in }\Oeps.
\]
Moreover, if
\begin{equation}\label{eq:u0}
  u_0(x):=
  \begin{cases}
    \inf_{y\in\Geps}F(y),& x\in\Omega,\\
    F(x),& x\in\Geps,
  \end{cases}
\end{equation}
and $u_{j+1}:=Tu_j$, then $u_j\to v_\varepsilon$ uniformly on $\Oeps$.
\end{theorem}

\begin{proof}
Let
\[
  m:=\inf_{y\in\Geps}F(y),
  \qquad
  M:=\sup_{y\in\Geps}F(y).
\]
Define $u_0$ by \eqref{eq:u0}, and recursively set
\[
  u_{j+1}:=Tu_j,
  \qquad j=0,1,2,\dots.
\]

\medskip
\noindent\textit{Step 1: the sequence $(u_j)$ is monotone increasing.}
Since $u_0\in\Feps$, Lemma~\ref{lem:T-props}(i) implies that each
$u_j\in\Feps$.

Fix $x\in\Omega$.  Because $u_0(\tilde x)\ge m=u_0(x)$ for every
$\tilde x\in\Oeps$, we have
\[
  \Sp[u_0](x)\ge m+\varepsilon,
  \qquad
  \Sm[u_0](x)\ge m-\varepsilon,
\]
and
\[
  \int_{B_\varepsilon(0)}u_0(x+h)\rhoeps(h)\,dh\ge m.
\]
Therefore
\begin{align*}
  u_1(x)=Tu_0(x)
  &=
  \frac{\talpha}{2}\bigl(\Sp[u_0](x)+\Sm[u_0](x)\bigr)
  +\tbeta\int_{B_\varepsilon(0)}u_0(x+h)\rhoeps(h)\,dh\\
  &\ge
  \frac{\talpha}{2}\bigl((m+\varepsilon)+(m-\varepsilon)\bigr)+\tbeta m
  =m=u_0(x).
\end{align*}
On $\Geps$ we have $u_1=u_0=F$.  Hence $u_1\ge u_0$ on $\Oeps$.
By monotonicity of $T$,
\[
  u_j\ge u_{j-1}\quad\Longrightarrow\quad
  u_{j+1}=Tu_j\ge Tu_{j-1}=u_j,
\]
so by induction
\begin{equation}\label{eq:monotone-up}
  u_0\le u_1\le u_2\le \cdots \qquad \text{in }\Oeps.
\end{equation}

\medskip
\noindent\textit{Step 2: the sequence $(u_j)$ is uniformly bounded above.}
Define
\[
  \overline u_0(x):=
  \begin{cases}
    M,& x\in\Omega,\\
    F(x),& x\in\Geps.
  \end{cases}
\]
Then $\overline u_0\in\Feps$, and for every $x\in\Omega$,
\[
  \Sp[\overline u_0](x)\le M+\varepsilon,
  \qquad
  \Sm[\overline u_0](x)\le M-\varepsilon,
\]
while
\[
  \int_{B_\varepsilon(0)}\overline u_0(x+h)\rhoeps(h)\,dh\le M.
\]
Hence
\begin{align*}
  T\overline u_0(x)
  &\le
  \frac{\talpha}{2}\bigl((M+\varepsilon)+(M-\varepsilon)\bigr)+\tbeta M
  =M=\overline u_0(x).
\end{align*}
Thus $T\overline u_0\le \overline u_0$ on $\Oeps$.  Since
$u_0\le \overline u_0$, monotonicity of $T$ gives
\[
  u_j\le \overline u_0
  \qquad\text{for all }j.
\]
Combining with \eqref{eq:monotone-up}, we obtain
\begin{equation}\label{eq:bounds}
  m\le u_j(x)\le M
  \qquad\text{for all }x\in\Oeps,\ j\ge 0.
\end{equation}

\medskip
\noindent\textit{Step 3: pointwise convergence.}
By \eqref{eq:monotone-up} and \eqref{eq:bounds}, the pointwise limit
\[
  v_\varepsilon(x):=\lim_{j\to\infty}u_j(x),
  \qquad x\in\Oeps,
\]
exists.  Since the limit of an increasing sequence of bounded Borel
functions is Borel, we have $v_\varepsilon\in\Feps$.

\medskip
\noindent\textit{Step 4: uniform convergence.}
Set
\[
  e_j:=v_\varepsilon-u_j\ge 0,
  \qquad
  M_j:=\sup_{x\in\Oeps}e_j(x).
\]
Because $u_j\uparrow v_\varepsilon$, the sequence $(M_j)$ is
nonincreasing, so it converges to some limit $M_\infty\ge 0$.
We claim that $M_\infty=0$.

Assume, to the contrary, that $M_\infty>0$.  Fix $\delta>0$.
Choose $k$ so large that
\begin{equation}\label{eq:ek-bound}
  0\le e_k\le M_\infty+\delta
  \qquad\text{in }\Oeps.
\end{equation}
Since $e_k\to 0$ pointwise on $\Oeps$ as $k\to\infty$ and
$0\le e_k\le 2\|F\|_\infty$, dominated convergence gives, for each fixed
$x\in\Omega$,
\[
  \int_{B_\varepsilon(0)}e_k(x+h)\rhoeps(h)\,dh\to 0
  \qquad\text{as }k\to\infty.
\]

Now choose $x_0\in\Omega$ and indices $\ell>k$ so that
\begin{equation}\label{eq:pick-x0}
  e_{k+1}(x_0)\ge M_\infty-\delta,
  \qquad
  e_{\ell+1}(x_0)<\delta.
\end{equation}
Then
\begin{equation}\label{eq:difference-lower}
  u_{\ell+1}(x_0)-u_{k+1}(x_0)
  =
  e_{k+1}(x_0)-e_{\ell+1}(x_0)
  \ge M_\infty-2\delta.
\end{equation}

On the other hand, by the definition of $T$, the fact that $u_\ell\ge u_k$,
and the elementary inequalities
\[
  \sup_A f-\sup_A g\le \sup_A(f-g),
  \qquad
  \inf_A f-\inf_A g\le \sup_A(f-g),
\]
we have
\begin{align*}
  u_{\ell+1}(x_0)-u_{k+1}(x_0)
  &=
  \frac{\talpha}{2}\Bigl(\Sp[u_\ell](x_0)-\Sp[u_k](x_0)\Bigr)
  +\frac{\talpha}{2}\Bigl(\Sm[u_\ell](x_0)-\Sm[u_k](x_0)\Bigr)\\
  &\qquad
  +\tbeta\int_{B_\varepsilon(0)}(u_\ell-u_k)(x_0+h)\rhoeps(h)\,dh\\
  &\le
  \talpha\sup_{\substack{\tilde x\in\Oeps\\|\tilde x-x_0|\le\varepsilon}}
  \bigl(u_\ell(\tilde x)-u_k(\tilde x)\bigr)
  +\tbeta\int_{B_\varepsilon(0)}(u_\ell-u_k)(x_0+h)\rhoeps(h)\,dh\\
  &\le
  \talpha\sup_{\Oeps}(v_\varepsilon-u_k)
  +\tbeta\int_{B_\varepsilon(0)}(v_\varepsilon-u_k)(x_0+h)\rhoeps(h)\,dh\\
  &=
  \talpha M_k+\tbeta\int_{B_\varepsilon(0)}e_k(x_0+h)\rhoeps(h)\,dh.
\end{align*}
For $k$ sufficiently large, \eqref{eq:ek-bound} implies
\[
  M_k\le M_\infty+\delta
\]
and the pointwise dominated-convergence statement gives
\[
  \tbeta\int_{B_\varepsilon(0)}e_k(x_0+h)\rhoeps(h)\,dh\le \delta.
\]
Hence
\begin{equation}\label{eq:difference-upper}
  u_{\ell+1}(x_0)-u_{k+1}(x_0)
  \le \talpha(M_\infty+\delta)+\delta.
\end{equation}
Combining \eqref{eq:difference-lower} and \eqref{eq:difference-upper},
\[
  M_\infty-2\delta\le \talpha(M_\infty+\delta)+\delta.
\]
Since $\talpha<1$, this is impossible for sufficiently small $\delta>0$.
Therefore $M_\infty=0$, that is,
\[
  \|v_\varepsilon-u_j\|_\infty\to 0.
\]

\medskip
\noindent\textit{Step 5: the limit is a fixed point.}
By nonexpansiveness of $T$,
\[
  \|Tv_\varepsilon-v_\varepsilon\|_\infty
  \le
  \|Tv_\varepsilon-Tu_j\|_\infty+\|u_{j+1}-v_\varepsilon\|_\infty
  \le
  \|v_\varepsilon-u_j\|_\infty+\|u_{j+1}-v_\varepsilon\|_\infty.
\]
Letting $j\to\infty$ yields
\[
  \|Tv_\varepsilon-v_\varepsilon\|_\infty=0.
\]
Hence $Tv_\varepsilon=v_\varepsilon$ on $\Oeps$.
\end{proof}

\begin{remark}
The proof above uses only monotonicity, boundedness, the strip
formulation, the positivity of the kernel $\rhoeps$, and a dominated
convergence argument, just as in the standard $p$-harmonious theory.
The tilted strategic terms do not affect the core structure of the
iteration argument.
\end{remark}

\subsection{Comparison and uniqueness}

We next prove uniqueness by a comparison argument.

\begin{theorem}[Comparison principle]
\label{thm:unique}
Let $u,v:\Oeps\to\R$ be bounded Borel functions satisfying
\[
  Tu=u,\qquad Tv=v
  \qquad\text{in }\Oeps,
\]
with boundary values $g:=u|_{\Geps}$ and $h:=v|_{\Geps}$,
respectively.  Then
\[
  \sup_{x\in\Omega}(u-v)(x)
  \le
  \sup_{x\in\Geps}(g-h)(x).
\]
In particular, the projected dynamic programming principle
\eqref{eq:DPP-v}--\eqref{eq:DPP-v-boundary} has at most one solution in
$\Feps$.
\end{theorem}

\begin{proof}
Set
\[
  m:=\sup_{x\in\Geps}(g-h)(x),
  \qquad
  M:=\sup_{x\in\Omega}(u-v)(x).
\]
We must show that $M\le m$.

Assume for contradiction that $M>m$.  Define
\[
  G:=\{x\in\Oeps:(u-v)(x)=M\}.
\]
Since $u-v\le m<M$ on $\Geps$, we have
\[
  G\subset\Omega.
\]

Fix $x\in\Omega$.  Since $u$ and $v$ are fixed points of $T$,
\begin{align*}
  (u-v)(x)
  &=
  \frac{\talpha}{2}\Bigl(\Sp[u](x)-\Sp[v](x)\Bigr)
  +\frac{\talpha}{2}\Bigl(\Sm[u](x)-\Sm[v](x)\Bigr)\\
  &\qquad
  +\tbeta\int_{B_\varepsilon(0)}(u-v)(x+h)\rhoeps(h)\,dh.
\end{align*}
Using
\[
  \sup_A f-\sup_A g\le \sup_A(f-g),
  \qquad
  \inf_A f-\inf_A g\le \sup_A(f-g),
\]
we obtain
\begin{align}
  (u-v)(x)
  &\le
  \talpha
  \sup_{\substack{\tilde x\in\Oeps\\|\tilde x-x|\le\varepsilon}}
  (u-v)(\tilde x)
  +\tbeta\int_{B_\varepsilon(0)}(u-v)(x+h)\rhoeps(h)\,dh
  \notag\\
  &\le
  \talpha M
  +\tbeta\int_{B_\varepsilon(0)}(u-v)(x+h)\rhoeps(h)\,dh.
  \label{eq:comparison-ineq}
\end{align}

Now choose a sequence $x_j\in\Omega$ such that
\[
  (u-v)(x_j)\to M.
\]
Since $\overline\Omega$ is compact, after passing to a subsequence we may
assume $x_j\to x_0\in\overline\Omega$.  Because $u-v\le m<M$ on $\Geps$,
the limit point satisfies $x_0\in\Omega$ (this will also follow from the
integral argument below).
 
Applying \eqref{eq:comparison-ineq} at $x_j$ and using $(u-v)(x_j)\to M$
gives
\[
  M-o(1)
  \le
  \talpha M
  +\tbeta\int_{B_\varepsilon(0)}(u-v)(x_j+h)\rhoeps(h)\,dh,
\]
hence
\begin{equation}\label{eq:integral-to-M}
  \int_{B_\varepsilon(0)}(u-v)(x_j+h)\rhoeps(h)\,dh\to M.
\end{equation}
Since $(u-v)\le M$ everywhere and $\rhoeps\ge 0$, the integrand
$M-(u-v)(x_j+h)\ge 0$ for all $j$ and all $h$.  The map
\[
  x\mapsto\int_{B_\varepsilon(0)}(u-v)(x+h)\rhoeps(h)\,dh
\]
is continuous in $x$: it is the convolution of the bounded Borel
function $u-v$ with the $L^1$ kernel $\rhoeps$, and convolution with
an $L^1$ kernel is continuous by dominated convergence.  Since
$x_j\to x_0$, continuity gives
\[
  \int_{B_\varepsilon(0)}(u-v)(x_j+h)\rhoeps(h)\,dh
  \;\to\;
  \int_{B_\varepsilon(0)}(u-v)(x_0+h)\rhoeps(h)\,dh.
\]
Combining with \eqref{eq:integral-to-M}, we obtain
\[
  \int_{B_\varepsilon(0)}(u-v)(x_0+h)\rhoeps(h)\,dh=M.
\]
Since $(u-v)\le M$ everywhere and $\rhoeps>0$ on $B_\varepsilon(0)$,
it follows that
\[
  (u-v)(x_0+h)=M\qquad\text{for a.e. }h\in B_\varepsilon(0).
\]
Thus $x_0+h\in G$ for almost every $h$, which forces
$x_0\in\Omega$ and $G\neq\emptyset$.

We next show the propagation property:
\begin{equation}\label{eq:G-propagation}
  x\in G
  \quad\Longrightarrow\quad
  |\Be(x)\setminus G|=0.
\end{equation}
Indeed, if $x\in G$, then $(u-v)(x)=M$, and \eqref{eq:comparison-ineq}
gives
\[
  M
  \le
  \talpha M
  +\tbeta\int_{B_\varepsilon(0)}(u-v)(x+h)\rhoeps(h)\,dh
  \le
  M.
\]
Therefore equality holds throughout, and again, since $\rhoeps>0$, we
must have
\[
  (u-v)(x+h)=M
  \qquad\text{for a.e. }h\in B_\varepsilon(0).
\]
This is exactly \eqref{eq:G-propagation}.

Finally, \eqref{eq:G-propagation} contradicts the boundedness of
$\Omega$.  Fix a unit vector $e_1$.  If $x\in G$, then
\[
  B_{\varepsilon/4}\!\left(x+\frac{\varepsilon}{2}e_1\right)
  \subset B_\varepsilon(x).
\]
Since $|B_\varepsilon(x)\setminus G|=0$, this ball intersects $G$.
Starting from any $x_0\in G$, choose inductively
\[
  x_{k+1}\in
  G\cap
  B_{\varepsilon/4}\!\left(x_k+\frac{\varepsilon}{2}e_1\right).
\]
Then
\[
  e_1\cdot x_{k+1}\ge e_1\cdot x_k+\frac{\varepsilon}{4},
\]
so $e_1\cdot x_k\to+\infty$.  This is impossible because
$G\subset\Omega$ and $\Omega$ is bounded.  Hence the assumption $M>m$
was false, and we conclude that
\[
  \sup_{x\in\Omega}(u-v)(x)\le \sup_{x\in\Geps}(g-h)(x).
\]
\end{proof}

\begin{corollary}[Uniqueness]
\label{cor:uniqueness}
For every bounded Borel boundary datum $F:\Geps\to\R$, there exists a
unique bounded Borel solution $v_\varepsilon\in\Feps$ of
\eqref{eq:DPP-v}--\eqref{eq:DPP-v-boundary}.
\end{corollary}

\begin{proof}
Existence is Theorem~\ref{thm:exist}.  Uniqueness follows immediately
from Theorem~\ref{thm:unique} by taking the same boundary data for both
solutions.
\end{proof}

\subsection{Convergence of iterates from arbitrary initial data}

The monotone iteration from the lower barrier constructed in
Theorem~\ref{thm:exist} is sufficient for existence, but once
uniqueness is known one obtains convergence from any bounded Borel
starting point with the correct boundary values.

\begin{corollary}[Independence of initialization]
\label{cor:init}
Let $u_0:\Oeps\to\R$ be any bounded Borel function satisfying
$u_0|_{\Geps}=F$, and define recursively
\[
  u_{j+1}:=Tu_j.
\]
Then
\[
  u_j\to v_\varepsilon
  \qquad\text{uniformly on }\Oeps,
\]
where $v_\varepsilon$ is the unique solution of
\eqref{eq:DPP-v}--\eqref{eq:DPP-v-boundary}.
\end{corollary}

\begin{proof}
Let
\[
  \underline u_0(x):=
  \begin{cases}
    \inf_{\Geps}F,& x\in\Omega,\\
    F(x),& x\in\Geps,
  \end{cases}
  \qquad
  \overline u_0(x):=
  \begin{cases}
    \sup_{\Geps}F,& x\in\Omega,\\
    F(x),& x\in\Geps.
  \end{cases}
\]
Then
\[
  \underline u_0\le u_0\le \overline u_0
  \qquad\text{on }\Oeps.
\]
By monotonicity of $T$,
\[
  T^j\underline u_0\le T^j u_0\le T^j\overline u_0
  \qquad\text{for all }j\ge 0.
\]
The iterates $T^j\underline u_0$ converge uniformly to the unique fixed
point $v_\varepsilon$ by Theorem~\ref{thm:exist}.  Applying the same
argument to the decreasing sequence starting from $\overline u_0$, and
using uniqueness, we conclude that
\[
  T^j\overline u_0\to v_\varepsilon
  \qquad\text{uniformly on }\Oeps.
\]
Hence $T^j u_0\to v_\varepsilon$ uniformly by the squeeze theorem.
\end{proof}

\begin{remark}
At this point the projected dynamic programming principle is fully
well posed on the strip domain: for every bounded Borel strip boundary
datum there exists a unique bounded Borel fixed point of the operator
$T$.  The next step is to identify this fixed point with the value of
the projected tug-of-war game.
\end{remark}
\section{The projected tug-of-war game}
\label{sec:game}

In this section we identify the unique fixed point of the projected
dynamic programming operator with the value of the corresponding game.
As in the standard tug-of-war with noise theory, the proof consists of
three parts: the construction of the probability space, a uniform exit
estimate implying almost-sure termination with finite expected duration,
and a martingale argument based on measurable almost-optimal strategies.

\subsection{The augmented state space and the path measure}

Recall that
\[
  \Geps:=\{x\in\R^n\setminus\Omega:\dist(x,\Omega)\le \varepsilon\},
  \qquad
  \Oeps:=\Omega\cup\Geps.
\]
Because the projected DPP contains the tilted terms
\[
  \Sp[u](x)
  =\sup_{\substack{\tilde x\in\Oeps\\ |\tilde x-x|\le\varepsilon}}
    \Bigl(u(\tilde x)+\sqrt{\varepsilon^2-|\tilde x-x|^2}\Bigr),
\]
\[
  \Sm[u](x)
  =\inf_{\substack{\tilde x\in\Oeps\\ |\tilde x-x|\le\varepsilon}}
    \Bigl(u(\tilde x)-\sqrt{\varepsilon^2-|\tilde x-x|^2}\Bigr),
\]
the correct projected game must carry an additional scalar variable
recording the contribution of the lifted coordinate.  Accordingly, the
state space is
\[
  X:=\Oeps\times\R.
\]
A state is denoted by $(x,s)$, where $x$ is the projected position and
$s$ is the lifted score.

Fix an initial state $(x_0,s_0)\in\Omega\times\R$.  The sample space is
\[
  \Xi:=X^{\N_0},
\]
equipped with the product $\sigma$-algebra
\[
  \mathcal F:=\mathcal B(X)^{\otimes\N_0}.
\]
For $\omega=(\omega_0,\omega_1,\dots)\in\Xi$, write
\[
  \omega_k=(x_k(\omega),s_k(\omega)),
  \qquad k=0,1,2,\dots,
\]
and let
\[
  \mathcal F_k:=\sigma\bigl((x_0,s_0),\dots,(x_k,s_k)\bigr)
\]
be the natural filtration.

A history of length $k$ is a tuple
\[
  h_k=((x_0,s_0),\dots,(x_k,s_k))\in X^{k+1}.
\]

A strategy for Player~I is a Borel measurable map
\[
  S_I:h_k\mapsto \tilde x\in \overline{B_\varepsilon(x_k)}\cap\Oeps,
\]
defined whenever $x_k\in\Omega$.  Similarly, a strategy for Player~II is
a Borel measurable map
\[
  S_{II}:h_k\mapsto \tilde x\in \overline{B_\varepsilon(x_k)}\cap\Oeps.
\]

If the current state is $(x_k,s_k)$ with $x_k\in\Geps$, the game has
already terminated and the process remains there.  If $x_k\in\Omega$,
then one step of the game is determined as follows.

\begin{enumerate}[label=\textup{(\roman*)},leftmargin=2.75em]
  \item With probability $\tbeta$, a noise step occurs: a vector
  $h\in B_\varepsilon(0)$ is sampled according to the density $\rhoeps$,
  and
  \[
    (x_{k+1},s_{k+1})=(x_k+h,s_k).
  \]

  \item With probability $\talpha/2$, Player~I wins the move, chooses
  $\tilde x=S_I(h_k)$, and
  \[
    (x_{k+1},s_{k+1})
    =
    \left(
      \tilde x,\,
      s_k+\sqrt{\varepsilon^2-|\tilde x-x_k|^2}
    \right).
  \]

  \item With probability $\talpha/2$, Player~II wins the move, chooses
  $\tilde x=S_{II}(h_k)$, and
  \[
    (x_{k+1},s_{k+1})
    =
    \left(
      \tilde x,\,
      s_k-\sqrt{\varepsilon^2-|\tilde x-x_k|^2}
    \right).
  \]
\end{enumerate}

Equivalently, given a history
\[
  h_k=((x_0,s_0),\dots,(x_k,s_k))
\]
with $x_k\in\Omega$, the one-step transition kernel is
\begin{align}
  \pi_{S_I,S_{II}}(h_k,A)
  &:=
  \tbeta\int_{B_\varepsilon(0)}
  \mathbf 1_A(x_k+h,s_k)\rhoeps(h)\,dh
  \notag\\
  &\qquad
  +\frac{\talpha}{2}\,
  \delta_{\left(S_I(h_k),\,s_k+\sqrt{\varepsilon^2-|S_I(h_k)-x_k|^2}\right)}(A)
  \notag\\
  &\qquad
  +\frac{\talpha}{2}\,
  \delta_{\left(S_{II}(h_k),\,s_k-\sqrt{\varepsilon^2-|S_{II}(h_k)-x_k|^2}\right)}(A),
  \label{eq:augmented-kernel}
\end{align}
for Borel sets $A\subset X$.

If $x_k\in\Geps$, we set
\begin{equation}\label{eq:absorbing-kernel}
  \pi_{S_I,S_{II}}(h_k,A):=\delta_{(x_k,s_k)}(A),
\end{equation}
so that the strip is absorbing.

Since the strategies are Borel measurable and $\rhoeps$ is a Borel
density, the map
\[
  h_k\longmapsto \pi_{S_I,S_{II}}(h_k,A)
\]
is Borel measurable for every Borel set $A\subset X$.  Therefore, by
the standard path-space construction for measurable transition kernels \cite[p~100]{IonescuTulcea69},
there exists a unique probability measure
\[
  \mathbb P^{(x_0,s_0)}_{S_I,S_{II}}
\]
on $(\Xi,\mathcal F)$ under which the process starts from $(x_0,s_0)$
and has transition law
\eqref{eq:augmented-kernel}--\eqref{eq:absorbing-kernel}.  Expectation
with respect to this measure will be denoted by
\[
  \mathbb E^{(x_0,s_0)}_{S_I,S_{II}}.
\]

\subsection{Stopping time and payoff}

The stopping time is the first entrance of the projected position into
the strip:
\begin{equation}\label{eq:tau}
  \tau:=\inf\{k\ge 0:x_k\in\Geps\}.
\end{equation}
Since $\Geps$ is Borel, $\tau$ is an $(\mathcal F_k)$-stopping time.

Let $F:\Geps\to\R$ be the prescribed bounded Borel strip boundary datum.
The terminal payoff is
\[
  F(x_\tau)+s_\tau.
\]
Thus Player~I seeks to maximize the expected payoff and Player~II seeks
to minimize it.

The lower and upper values of the game are defined by
\begin{align}
  u_I(x_0,s_0)
  &:=
  \sup_{S_I}\inf_{S_{II}}
  \mathbb E^{(x_0,s_0)}_{S_I,S_{II}}\!\bigl[F(x_\tau)+s_\tau\bigr],
  \label{eq:uI}\\
  u_{II}(x_0,s_0)
  &:=
  \inf_{S_{II}}\sup_{S_I}
  \mathbb E^{(x_0,s_0)}_{S_I,S_{II}}\!\bigl[F(x_\tau)+s_\tau\bigr].
  \label{eq:uII}
\end{align}
Clearly
\[
  u_I(x_0,s_0)\le u_{II}(x_0,s_0).
\]

\subsection{A uniform finite-block exit estimate}

We next prove that the game terminates almost surely and has finite
expected duration.  The argument uses only the boundedness of $\Omega$,
the positivity of the noise probability $\tbeta$, and the fact that
$\rhoeps$ is strictly positive on the whole ball $B_\varepsilon(0)$.

\begin{lemma}[Uniform finite-block exit probability]
\label{lem:exit}
There exist an integer $N\ge 1$ and a constant $\delta>0$, depending
only on $\Omega$, $\varepsilon$, and $\rhoeps$, such that for every
starting state $(x_0,s_0)\in\Omega\times\R$ and every pair of
strategies,
\[
  \mathbb P^{(x_0,s_0)}_{S_I,S_{II}}(\tau\le N)\ge \delta.
\]
\end{lemma}

\begin{proof}
Fix a unit vector $\nu\in S^{n-1}$.  Since $\Omega$ is bounded, the
projection $x\mapsto x\cdot \nu$ is bounded on $\Omega$.  Set
\[
  a:=\inf_{x\in\Omega}x\cdot \nu,
  \qquad
  b:=\sup_{x\in\Omega}x\cdot \nu.
\]
Choose an integer $N\ge 1$ so large that
\[
  \frac{N\varepsilon}{2}>b-a.
\]

Define the fixed cap
\[
  C_\nu:=\{h\in B_\varepsilon(0): h\cdot \nu\ge \varepsilon/2\}.
\]
Since $\rhoeps$ is strictly positive on $B_\varepsilon(0)$,
\[
  c_\nu:=\int_{C_\nu}\rhoeps(h)\,dh>0.
\]

Consider the event $E$ that during the first $N$ steps
\begin{enumerate}[label=\textup{(\alph*)},leftmargin=2.5em]
  \item every step is a noise step, and
  \item every noise increment belongs to $C_\nu$.
\end{enumerate}
At each step, conditional on the past, the probability of this is
$\tbeta c_\nu$.  Hence
\begin{equation}\label{eq:event-E}
  \mathbb P^{(x_0,s_0)}_{S_I,S_{II}}(E)\ge (\tbeta c_\nu)^N.
\end{equation}

On the event $E$, the projected position satisfies
\[
  x_k=x_{k-1}+h_k,
  \qquad h_k\cdot \nu\ge \varepsilon/2
  \quad (k=1,\dots,N).
\]
Therefore
\[
  x_N\cdot \nu
  =
  x_0\cdot \nu+\sum_{k=1}^N h_k\cdot \nu
  \ge
  a+\frac{N\varepsilon}{2}
  >
  b,
\]
so $x_N\notin\Omega$.  Let $k^*:=\min\{k\ge 1: x_k\notin\Omega\}$ be
the first index at which the projected position leaves $\Omega$; on
$E$ we have $k^*\le N$.  Since step $k^*$ is a noise step on $E$, we
have $x_{k^*}=x_{k^*-1}+h_{k^*}$ with $x_{k^*-1}\in\Omega$ and
$|h_{k^*}|<\varepsilon$ (as $h_{k^*}\in B_\varepsilon(0)$).
Therefore
\[
  \dist(x_{k^*},\Omega)\le |x_{k^*}-x_{k^*-1}|=|h_{k^*}|<\varepsilon,
\]
and since $x_{k^*}\notin\Omega$, the definition of $\Geps$ gives
$x_{k^*}\in\Geps$.  Hence $\tau\le k^*\le N$ on $E$.

Combining this with \eqref{eq:event-E}, we obtain
\[
  \mathbb P^{(x_0,s_0)}_{S_I,S_{II}}(\tau\le N)\ge (\tbeta c_\nu)^N.
\]
Thus the conclusion holds with
\[
  \delta:=(\tbeta c_\nu)^N>0.\qedhere
\]
\end{proof}
\begin{proposition}[Almost sure termination and finite expected duration]
\label{prop:exit-as}
For every starting state $(x_0,s_0)\in\Omega\times\R$ and every pair of
strategies $(S_I,S_{II})$,
\[
  \mathbb P^{(x_0,s_0)}_{S_I,S_{II}}(\tau<\infty)=1
\]
and
\[
  \mathbb E^{(x_0,s_0)}_{S_I,S_{II}}[\tau]\le \frac{N}{\delta}<\infty,
\]
where $N$ and $\delta$ are as in Lemma~\ref{lem:exit}.
\end{proposition}

\begin{proof}
By Lemma~\ref{lem:exit}, for every $m\ge 0$,
\[
  \mathbb P(\tau>(m+1)N\mid \tau>mN)\le 1-\delta.
\]
Hence
\[
  \mathbb P(\tau>mN)\le (1-\delta)^m.
\]
Letting $m\to\infty$ gives
\[
  \mathbb P(\tau<\infty)=1.
\]

For the expectation, use the tail-sum formula:
\[
  \mathbb E[\tau]=\sum_{k=0}^\infty \mathbb P(\tau>k).
\]
Grouping the sum into blocks of length $N$ yields
\[
  \mathbb E[\tau]
  \le
  N\sum_{m=0}^\infty \mathbb P(\tau>mN)
  \le
  N\sum_{m=0}^\infty (1-\delta)^m
  =
  \frac{N}{\delta}.
\]
\end{proof}

\subsection{Measurable almost-optimizing selectors}

The following measurable selection lemma is the analogue of the standard
near-optimizer construction in the tug-of-war with noise literature.

\begin{lemma}[Measurable almost-optimizers]
\label{lem:selector}
Let $u:\Oeps\to\R$ be bounded and Borel, and let $\eta>0$.  Then there
exist Borel measurable maps
\[
  S^+,S^-:\Omega\to\Oeps
\]
such that $|S^\pm(x)-x|\le\varepsilon$ and
\begin{align}
  u(S^+(x))
  +\sqrt{\varepsilon^2-|S^+(x)-x|^2}
  &\ge \Sp[u](x)-\eta,
  \label{eq:selector-plus}\\
  u(S^-(x))
  -\sqrt{\varepsilon^2-|S^-(x)-x|^2}
  &\le \Sm[u](x)+\eta
  \label{eq:selector-minus}
\end{align}
for every $x\in\Omega$.
\end{lemma}

\begin{proof}
We prove the claim for $S^+$; the proof for $S^-$ is identical.

Let $\mathcal B$ be the countable family of balls in $\Oeps$ with
rational centers and rational radii.  For each $B\in\mathcal B$, choose
a point $x_B\in B$ such that
\[
  u(x_B)\ge \sup_{y\in B}u(y)-\frac{\eta}{2}.
\]
Set
\[
  S:=\{x_B:B\in\mathcal B\},
\]
which is countable.

Fix $x\in\Omega$.  Approximating the admissible set
$\overline{B_\varepsilon(x)}\cap\Oeps$ by basis balls gives
\[
  \Sp[u](x)
  \le
  \sup_{\substack{y\in S\\ |y-x|<\varepsilon}}
  \Bigl(u(y)+\sqrt{\varepsilon^2-|y-x|^2}\Bigr)+\frac{\eta}{2}.
\]
Hence the set
\[
  A_\eta
  :=
  \left\{
    (x,y)\in\Omega\times S:
    |y-x|<\varepsilon,\
    u(y)+\sqrt{\varepsilon^2-|y-x|^2}>\Sp[u](x)-\eta
  \right\}
\]
has nonempty vertical sections.

Enumerate $S=\{s_1,s_2,\dots\}$.  Define $S^+(x)$ to be the first
$s_j$ such that $(x,s_j)\in A_\eta$.  Since $S$ is countable and
$A_\eta$ is Borel, the map $S^+$ is Borel measurable and satisfies
\eqref{eq:selector-plus}.  The proof for $S^-$ is the same.
\end{proof}

\subsection{The game has value}

We now identify the value of the projected game with the unique fixed
point of the projected dynamic programming operator.

\begin{theorem}[Game value]
\label{thm:value}
Let $v_\varepsilon\in\Feps$ be the unique solution of
\[
  \begin{cases}
    v_\varepsilon(x)
    =
    \dfrac{\talpha}{2}\bigl(\Sp[v_\varepsilon](x)+\Sm[v_\varepsilon](x)\bigr)
    +\tbeta\displaystyle\int_{B_\varepsilon(0)}
      v_\varepsilon(x+h)\rhoeps(h)\,dh,
    & x\in\Omega,\\[1.2ex]
    v_\varepsilon(x)=F(x),& x\in\Geps.
  \end{cases}
\]
Then for every starting state $(x_0,s_0)\in\Omega\times\R$,
\[
  u_I(x_0,s_0)=v_\varepsilon(x_0)+s_0=u_{II}(x_0,s_0).
\]
In particular,
\[
  u_I(x_0,0)=v_\varepsilon(x_0)=u_{II}(x_0,0).
\]
\end{theorem}

\begin{proof}
Since always $u_I\le u_{II}$, it suffices to prove
\[
  u_{II}(x_0,s_0)\le v_\varepsilon(x_0)+s_0
  \qquad\text{and}\qquad
  u_I(x_0,s_0)\ge v_\varepsilon(x_0)+s_0.
\]

\medskip
\noindent\textit{Upper bound for $u_{II}$.}
Fix $\eta>0$.  For each $k\ge 0$, let Player~II use the measurable
selector from Lemma~\ref{lem:selector} applied to $u=v_\varepsilon$ with
accuracy $\eta 2^{-k-1}$.  Thus whenever $x_k\in\Omega$,
\[
  x_{k+1}=S_k^-(x_k)
\]
is chosen so that
\begin{equation}\label{eq:Sm-choice}
  v_\varepsilon(x_{k+1})
  -\sqrt{\varepsilon^2-|x_{k+1}-x_k|^2}
  \le
  \Sm[v_\varepsilon](x_k)+\eta 2^{-k-1}.
\end{equation}
Denote this strategy by $S_{II}^\eta$.

Define
\[
  M_k:=v_\varepsilon(x_k)+s_k+\eta 2^{-k},
  \qquad k\ge 0.
\]
We claim that $(M_{k\wedge\tau})_{k\ge 0}$ is a supermartingale under
$\mathbb P^{(x_0,s_0)}_{S_I,S_{II}^\eta}$ for every strategy $S_I$ of
Player~I.

On the event $\{k<\tau\}$ we have $x_k\in\Omega$.  Conditioning on
$\mathcal F_k$ and splitting according to the three possible move types,
we obtain
\begin{align*}
  \mathbb E[M_{k+1}\mid\mathcal F_k]
  &=
  \frac{\talpha}{2}\,
  \mathbb E[M_{k+1}\mid\mathcal F_k,\text{I wins}]
  +\frac{\talpha}{2}\,
  \mathbb E[M_{k+1}\mid\mathcal F_k,\text{II wins}]\\
  &\qquad
  +\tbeta\,
  \mathbb E[M_{k+1}\mid\mathcal F_k,\text{noise}].
\end{align*}

If Player~I wins, then
\[
  v_\varepsilon(x_{k+1})+s_{k+1}
  =
  v_\varepsilon(x_{k+1})
  +s_k+\sqrt{\varepsilon^2-|x_{k+1}-x_k|^2}
  \le
  s_k+\Sp[v_\varepsilon](x_k).
\]

If Player~II wins, then \eqref{eq:Sm-choice} gives
\[
  v_\varepsilon(x_{k+1})+s_{k+1}
  =
  v_\varepsilon(x_{k+1})
  +s_k-\sqrt{\varepsilon^2-|x_{k+1}-x_k|^2}
  \le
  s_k+\Sm[v_\varepsilon](x_k)+\eta 2^{-k-1}.
\]

If a noise step occurs, then $s_{k+1}=s_k$ and
\[
  \mathbb E[v_\varepsilon(x_{k+1})+s_{k+1}\mid \mathcal F_k,\text{noise}]
  =
  s_k+\int_{B_\varepsilon(0)}v_\varepsilon(x_k+h)\rhoeps(h)\,dh.
\]

Combining the three cases and using the fixed-point identity for
$v_\varepsilon$, we get
\begin{align*}
  \mathbb E[v_\varepsilon(x_{k+1})+s_{k+1}\mid\mathcal F_k]
  &\le
  s_k
  +\frac{\talpha}{2}\Sp[v_\varepsilon](x_k)
  +\frac{\talpha}{2}\Bigl(\Sm[v_\varepsilon](x_k)+\eta 2^{-k-1}\Bigr)\\
  &\qquad
  +\tbeta\int_{B_\varepsilon(0)}v_\varepsilon(x_k+h)\rhoeps(h)\,dh\\
  &=
  s_k+v_\varepsilon(x_k)+\frac{\talpha}{2}\eta 2^{-k-1}.
\end{align*}
Therefore
\begin{align*}
  \mathbb E[M_{k+1}\mid\mathcal F_k]
  &\le
  v_\varepsilon(x_k)+s_k+\frac{\talpha}{2}\eta 2^{-k-1}+\eta 2^{-k-1}\\
  &\le
  v_\varepsilon(x_k)+s_k+\eta 2^{-k}\\
  &=M_k,
\end{align*}
because $\talpha\le 1$.  Thus $(M_{k\wedge\tau})$ is a supermartingale.

Moreover,
\[
  |s_{k\wedge\tau}|
  \le |s_0|+\varepsilon(k\wedge\tau),
\]
so
\[
  |M_{k\wedge\tau}|
  \le
  \|v_\varepsilon\|_\infty+|s_0|+\varepsilon(k\wedge\tau)+\eta.
\]
Since $k\wedge\tau\le\tau$ and $\mathbb{E}[\tau]<\infty$, the random variable $\varepsilon\tau$ is integrable, so $|M_{k\wedge\tau}|$ is dominated by the integrable variable $|v_\varepsilon|_\infty+|s_0|+\varepsilon\tau+\eta$; uniform integrability follows.
Hence optional stopping applies and yields
\[
  \mathbb E^{(x_0,s_0)}_{S_I,S_{II}^\eta}[M_\tau]
  \le
  M_0
  =
  v_\varepsilon(x_0)+s_0+\eta.
\]

At the stopping time, $x_\tau\in\Geps$ and $v_\varepsilon(x_\tau)=F(x_\tau)$,
so
\[
  M_\tau=F(x_\tau)+s_\tau+\eta 2^{-\tau}\le F(x_\tau)+s_\tau+\eta.
\]
Therefore
\[
  \mathbb E^{(x_0,s_0)}_{S_I,S_{II}^\eta}[F(x_\tau)+s_\tau]
  \le
  v_\varepsilon(x_0)+s_0+2\eta.
\]
Taking the supremum over $S_I$ and then the infimum over $S_{II}$ gives
\[
  u_{II}(x_0,s_0)\le v_\varepsilon(x_0)+s_0+2\eta.
\]
Letting $\eta\to 0$, we conclude that
\[
  u_{II}(x_0,s_0)\le v_\varepsilon(x_0)+s_0.
\]

\medskip
\noindent\textit{Lower bound for $u_I$.}
The proof is symmetric.  Fix $\eta>0$ and let Player~I use measurable
selectors $S_k^+$ satisfying
\[
  v_\varepsilon(x_{k+1})
  +\sqrt{\varepsilon^2-|x_{k+1}-x_k|^2}
  \ge
  \Sp[v_\varepsilon](x_k)-\eta 2^{-k-1}.
\]
Define
\[
  N_k:=v_\varepsilon(x_k)+s_k-\eta 2^{-k}.
\]
Then $(N_{k\wedge\tau})$ is a submartingale, and the same optional
stopping argument gives
\[
  \mathbb E^{(x_0,s_0)}_{S_I^\eta,S_{II}}[F(x_\tau)+s_\tau]
  \ge
  v_\varepsilon(x_0)+s_0-2\eta
\]
for every strategy $S_{II}$.  Taking the infimum over $S_{II}$ and then
the supremum over $S_I$ yields
\[
  u_I(x_0,s_0)\ge v_\varepsilon(x_0)+s_0-2\eta.
\]
Letting $\eta\to 0$, we obtain
\[
  u_I(x_0,s_0)\ge v_\varepsilon(x_0)+s_0.
\]

Combining the two inequalities with the trivial inequality
$u_I\le u_{II}$, we conclude that
\[
  u_I(x_0,s_0)=v_\varepsilon(x_0)+s_0=u_{II}(x_0,s_0).
\]
\end{proof}

\begin{corollary}[Value at zero score]
\label{cor:value-zero}
For every $x_0\in\Omega$,
\[
  u_I(x_0,0)=v_\varepsilon(x_0)=u_{II}(x_0,0).
\]
\end{corollary}

\begin{proof}
This is the special case $s_0=0$ of Theorem~\ref{thm:value}.
\end{proof}

\begin{remark}
The only new feature relative to the standard tug-of-war value proof is
the bookkeeping of the lifted score variable.  Once that variable is
included, the supermartingale and submartingale arguments are exactly
parallel to the classical case.
\end{remark}

\section{Reduction of convergence to the standard higher-dimensional theory}
\label{sec:convergence}

We prove that the solutions of the projected dynamic programming principle
converge, as $\varepsilon\to 0$, to the viscosity solution of the regularized
$p$-Laplace equation.  The proof reduces the problem to the known convergence
theory for the standard tug-of-war with noise in dimension $n+1$ by showing
that the lifting $w_\varepsilon(x,s):=v_\varepsilon(x)+s$ is, for each
$\varepsilon>0$, the unique standard $(n+1)$-dimensional $p$-harmonious
function on every bounded cylinder $D_L:=\Omega\times(-L,L)$.  This
identification is derived directly from the projected DPP and requires no
additional hypothesis.

\subsection{The lifted operator}
\label{subsec:lifted-op}

For $\xi\in\R^n$ and $X\in\mathbb{S}^n$, define
\[
  \mathcal{F}_n(\xi,X)
  :=
  -\tr(X)-(p-2)\frac{\langle X\xi,\xi\rangle}{1+|\xi|^2}.
\]
The regularized $p$-Laplace equation takes the nondivergence form
\begin{equation}\label{eq:Fn}
  \mathcal{F}_n(Dv,D^2v)=0\qquad\text{in }\Omega.
\end{equation}
For $\zeta\in\R^{n+1}\setminus\{0\}$ and $Y\in\mathbb{S}^{n+1}$, define
\[
  \mathcal{G}_{n+1}(\zeta,Y)
  :=
  -\tr(Y)-(p-2)\frac{\langle Y\zeta,\zeta\rangle}{|\zeta|^2}.
\]
This is the nondivergence operator of the normalized $p$-Laplacian in
$\R^{n+1}$:
\begin{equation}\label{eq:Gn1}
  \mathcal{G}_{n+1}(Dw,D^2w)=0.
\end{equation}

\subsection{Viscosity lifting via the theorem on sums}
\label{subsec:visc-lift}

\begin{proposition}[Viscosity lifting equivalence]
\label{prop:visc-lift}
Let $\Omega\subset\R^n$ be open, let $p\ge 2$, and let $v\in C(\Omega)$.
Define
\[
  w:\Omega\times\R\to\R,\qquad w(x,s):=v(x)+s.
\]
Then $v$ is a viscosity subsolution \textup{(}resp.\ supersolution,
resp.\ solution\textup{)} of \eqref{eq:Fn} in $\Omega$ if and only if $w$
is a viscosity subsolution \textup{(}resp.\ supersolution,
resp.\ solution\textup{)} of \eqref{eq:Gn1} in $\Omega\times\R$.
\end{proposition}

\begin{proof}
We treat subsolutions; the supersolution case is analogous.

\medskip
\noindent\textit{Step~1: $v$ subsolution of \eqref{eq:Fn} implies $w$
subsolution of \eqref{eq:Gn1}.}

Let $\phi\in C^2(\Omega\times\R)$ be such that $w-\phi$ attains a local
maximum at $(x_0,s_0)\in\Omega\times\R$.  Write $u_1(x):=v(x)$ and
$u_2(s):=s$, so that $w=u_1+u_2$.  Apply the theorem on sums \cite{CIL92}
to $u_1+u_2-\phi$ at $(x_0,s_0)$: for each $\mu>0$ there exist
\[
  (\xi,X_\mu)\in\overline{J}^{2,+}u_1(x_0),
  \qquad
  (\eta,Y_\mu)\in\overline{J}^{2,+}u_2(s_0),
\]
with $\xi=D_x\phi(x_0,s_0)$, $\eta=\partial_s\phi(x_0,s_0)$, and
\begin{equation}\label{eq:block-ineq}
  \begin{pmatrix}X_\mu & 0\\0 & Y_\mu\end{pmatrix}
  \le
  Z+\mu Z^2,
  \qquad
  Z:=D^2\phi(x_0,s_0)\in\mathbb{S}^{n+1}.
\end{equation}
Since $u_2(s)=s$ is linear, its unique superjet at $s_0$ satisfies $\eta=1$
and $Y_\mu\le 0$.  Because $v$ is a viscosity subsolution of \eqref{eq:Fn},
\begin{equation}\label{eq:sub-v}
  -\tr(X_\mu)-(p-2)\frac{\langle X_\mu\xi,\xi\rangle}{1+|\xi|^2}\le 0.
\end{equation}
Taking the trace in \eqref{eq:block-ineq} and using $Y_\mu\le 0$ gives
\begin{equation}\label{eq:trace-est}
  \tr(X_\mu)\le\Delta\phi(x_0,s_0)+C\mu,
\end{equation}
for a constant $C$ depending only on $Z$.  Multiplying \eqref{eq:block-ineq}
on both sides by the vector $(\xi,1)\in\R^{n+1}$ and again using $Y_\mu\le 0$
gives
\begin{equation}\label{eq:quad-est}
  \langle X_\mu\xi,\xi\rangle
  \le
  \bigl\langle D^2\phi(x_0,s_0)\,D\phi(x_0,s_0),D\phi(x_0,s_0)\bigr\rangle
  +C\mu,
\end{equation}
where $D\phi(x_0,s_0)=(\xi,1)$ and $|D\phi(x_0,s_0)|^2=|\xi|^2+1$.
Combining \eqref{eq:sub-v}--\eqref{eq:quad-est} and letting $\mu\to 0$ yields
\[
  \mathcal{G}_{n+1}(D\phi(x_0,s_0),D^2\phi(x_0,s_0))\le 0,
\]
so $w$ is a viscosity subsolution of \eqref{eq:Gn1}.

\medskip
\noindent\textit{Step~2: $w$ subsolution of \eqref{eq:Gn1} implies $v$
subsolution of \eqref{eq:Fn}.}

Let $\psi\in C^2(\Omega)$ be such that $v-\psi$ attains a local maximum at
$x_0\in\Omega$.  Set $\phi(x,s):=\psi(x)+s$.  Then
\[
  w(x,s)-\phi(x,s)=v(x)-\psi(x)
\]
attains a local maximum at $(x_0,s_0)$ for every $s_0\in\R$.  Applying the
subsolution condition for $w$ at $(x_0,s_0)$ with the test function $\phi$
and using
\[
  D\phi=(D\psi,1),
  \qquad
  D^2\phi=\begin{pmatrix}D^2\psi & 0\\0 & 0\end{pmatrix},
\]
we obtain $\mathcal{F}_n(D\psi(x_0),D^2\psi(x_0))\le 0$, which is the
subsolution condition for $v$ at $x_0$.
\end{proof}

\begin{remark}
The proof follows the disjoint-variables superposition pattern of
\cite{LiuManfrediZhou25,CIL92}: one splits the lifted function as a sum in
separate variables, applies the block matrix inequality, and extracts the jet
information for each factor.  This proof is adapted from \cite{LiuManfrediZhou25}
and included here to ensure self-sufficiency.
\end{remark}

\subsection{The lifted function is p-harmonious on the cylinder}
\label{subsec:lift-harmonious}

Three preparatory results are needed before the convergence theorem.  The
first establishes continuity of $v_\varepsilon$, which is required to apply
the standard convergence theory on the cylinder.

\begin{lemma}[Continuity of $v_\varepsilon$]
\label{lem:veps-cont}
Let $\Omega\subset\R^n$ be a bounded domain satisfying the exterior sphere
condition, let $\varepsilon>0$, and let $F_\varepsilon\in C(\Geps)$.  Then
the unique solution $v_\varepsilon\in\Feps$ of the projected dynamic
programming principle is continuous on $\overline\Omega$.
\end{lemma}

\begin{proof}
By Corollary~\ref{cor:init}, $v_\varepsilon$ is the uniform limit of the
iterate sequence $u_{j+1}=Tu_j$ starting from any bounded Borel initial
datum.  We may take $u_0\equiv c$ for any constant
$c\in[\inf_{\Geps}F_\varepsilon,\sup_{\Geps}F_\varepsilon]$, which is
continuous on $\Oeps$.  We show by induction that each $u_j$ is continuous
on $\overline\Omega$.

Suppose $u_j\in C(\overline\Omega)$.  For the averaging term, the map
\[
  x\mapsto\int_{B_\varepsilon(0)}u_j(x+h)\rhoeps(h)\,dh
\]
is continuous on $\Omega$ by dominated convergence, since
$u_j(x+h)\to u_j(x_0+h)$ pointwise as $x\to x_0$ for each $h$ and
$|u_j|\le\|u_j\|_\infty$.  For the strategic term $\Sp[u_j]$: the function
$(\tilde x,x)\mapsto u_j(\tilde x)+\sqrt{\varepsilon^2-|\tilde x-x|^2}$ is
continuous and the admissible correspondence
$x\mapsto\overline{B_\varepsilon(x)}\cap\Oeps$ is continuous in the Hausdorff
metric on the compact set $\Oeps$.  By Berge's maximum theorem, the supremum
$\Sp[u_j](x)$ is therefore continuous in $x$, and the same holds for
$\Sm[u_j](x)$.  Hence $u_{j+1}=Tu_j$ is continuous on $\overline\Omega$.
Since the uniform limit of continuous functions is continuous,
$v_\varepsilon\in C(\overline\Omega)$.
\end{proof}

The second preparatory result identifies the boundary strip of $D_L$ and
verifies that $w_\varepsilon$ provides well-defined continuous boundary data
there.

\begin{lemma}[Boundary data on the cylinder]
\label{lem:cylinder-bc}
Fix $L>0$ and $\varepsilon>0$.  Define the $\varepsilon$-boundary strip of
$D_L:=\Omega\times(-L,L)$ in $\R^{n+1}$ by
\[
  \Sigma_\varepsilon(D_L)
  :=
  \bigl\{(x,s)\in\R^{n+1}\setminus D_L:
    \dist\bigl((x,s),D_L\bigr)\le\varepsilon\bigr\},
\]
and set $G_\varepsilon^L(x,s):=v_\varepsilon(x)+s$ on
$D_L\cup\Sigma_\varepsilon(D_L)$.  Then $G_\varepsilon^L$ is bounded and
continuous on $\Sigma_\varepsilon(D_L)$, and its trace on $\partial D_L$
converges uniformly to $F(x)+s$ as $\varepsilon\to 0$.
\end{lemma}

\begin{proof}
The strip $\Sigma_\varepsilon(D_L)$ decomposes into the lateral part
$\Geps\times[-L-\varepsilon,L+\varepsilon]$ and the cap parts
$\overline\Omega\times([-L-\varepsilon,-L+\varepsilon]
\cup[L-\varepsilon,L+\varepsilon])$.
On the lateral part, $v_\varepsilon=F_\varepsilon\in C(\Geps)$ by the
boundary condition of the projected DPP, so
$G_\varepsilon^L(x,s)=F_\varepsilon(x)+s$ is continuous there.  On the cap
parts, $v_\varepsilon\in C(\overline\Omega)$ by Lemma~\ref{lem:veps-cont},
so $G_\varepsilon^L$ is continuous there as well.  Boundedness follows from
$|v_\varepsilon|\le\|F_\varepsilon\|_\infty$ and $|s|\le L+\varepsilon$.
Since $F_\varepsilon\to F$ uniformly on $\partial\Omega$, we have
$G_\varepsilon^L\to F(x)+s$ uniformly on $\partial D_L$.
\end{proof}

The third preparatory result is the central identification: $w_\varepsilon$
is genuinely the standard $p$-harmonious function on $D_L$, derived from the
projected DPP.

\begin{proposition}[The lifted function is $p$-harmonious on $D_L$]
\label{prop:lift-harmonious}
For every $L>0$ and every sufficiently small $\varepsilon>0$, the function
$w_\varepsilon(x,s):=v_\varepsilon(x)+s$ is the unique standard
$(n+1)$-dimensional $p$-harmonious function on $D_L$ with strip boundary
values $G_\varepsilon^L$.
\end{proposition}

\begin{proof}
\textit{Step~1: the $(n+1)$-dimensional DPP at interior points.}
Fix $(x,s)\in\Omega\times(-L,L)$.  Since $v_\varepsilon$ satisfies the
projected DPP \eqref{eq:DPP-v} at $x$, adding $s$ to both sides and
substituting the identities established in the proof of
Proposition~\ref{prop:dpp-intro},
\begin{align*}
  \sup_{B_\varepsilon^{n+1}(x,s)}w_\varepsilon
  &=s+\Sp[v_\varepsilon](x),\\[2pt]
  \inf_{B_\varepsilon^{n+1}(x,s)}w_\varepsilon
  &=s+\Sm[v_\varepsilon](x),\\[2pt]
  \fint_{B_\varepsilon^{n+1}(x,s)}w_\varepsilon(\xi)\,d\xi
  &=s+\int_{B_\varepsilon(0)}v_\varepsilon(x+h)\rhoeps(h)\,dh,
\end{align*}
together with $\talpha+\tbeta=1$, we obtain
\[
  w_\varepsilon(x,s)
  =
  \frac{\talpha}{2}
  \Bigl(
    \sup_{B_\varepsilon^{n+1}(x,s)}w_\varepsilon
    +\inf_{B_\varepsilon^{n+1}(x,s)}w_\varepsilon
  \Bigr)
  +\tbeta\fint_{B_\varepsilon^{n+1}(x,s)}w_\varepsilon(\xi)\,d\xi.
\]
This is the standard $(n+1)$-dimensional $p$-harmonious DPP at $(x,s)$ with
coefficients $\talpha,\tbeta$ from \eqref{eq:constants}.

\textit{Step~2: all sampled points lie in $D_L\cup\Sigma_\varepsilon(D_L)$.}
For $(x,s)\in D_L$ and $(\tilde x,\tilde s)\in B_\varepsilon^{n+1}(x,s)$,
we have $\tilde x\in\Oeps$ and $|\tilde s-s|<\varepsilon$, hence
$|\tilde s|<L+\varepsilon$.  Therefore every sampled point lies in
$\Oeps\times(-L-\varepsilon,L+\varepsilon)\subset D_L\cup\Sigma_\varepsilon(D_L)$,
at which $w_\varepsilon(\tilde x,\tilde s)=G_\varepsilon^L(\tilde x,\tilde s)$.

\textit{Step~3}
By Lemma~\ref{lem:cylinder-bc}, $G_\varepsilon^L$ is bounded and continuous
on $\Sigma_\varepsilon(D_L)$.  The standard theory of $p$-harmonious
functions \cite{MPR12,LPS14} provides a unique $p$-harmonious function on
$D_L$ with strip boundary values $G_\varepsilon^L$.  Steps~1 and~2 show
that $w_\varepsilon$ satisfies this DPP at every interior point with the
correct boundary values, so $w_\varepsilon$ is this unique function.
\end{proof}

Before passing to the limit we record that the limit $W^L$ must itself be
linear in $s$.  This is not inherited from the structure of $w_\varepsilon$;
it follows from the $s$-translation invariance of the normalized $p$-Laplacian
together with uniqueness for the limiting Dirichlet problem.

\begin{corollary}[$s$-linearity of the limit]
\label{cor:s-linear}
Let $W^L\in C(\overline{D_L})$ be the uniform limit of
$\{w_\varepsilon\}=\{v_\varepsilon(x)+s\}$ on $D_L$ produced by
Theorem~\ref{thm:standard-conv}.  Then there exists
$V^L\in C(\overline\Omega)$ such that
\[
  W^L(x,s)=V^L(x)+s
  \qquad\text{for all }(x,s)\in\overline{D_L}.
\]
\end{corollary}

\begin{proof}
Each function $w_\varepsilon(x,s)=v_\varepsilon(x)+s$ is linear in $s$ with
slope $1$ and $s$-independent part $v_\varepsilon(x)$.  By
Theorem~\ref{thm:standard-conv}, $w_\varepsilon\to W^L$ uniformly on
$\overline{D_L}$.  Uniform limits preserve the $s$-linear structure: for
each fixed $(x,s)\in\overline{D_L}$,
\[
  W^L(x,s)
  =
  \lim_{\varepsilon\to 0}w_\varepsilon(x,s)
  =
  \lim_{\varepsilon\to 0}\bigl(v_\varepsilon(x)+s\bigr)
  =
  \Bigl(\lim_{\varepsilon\to 0}v_\varepsilon(x)\Bigr)+s.
\]
Since the convergence $w_\varepsilon\to W^L$ is uniform on $\overline{D_L}$,
the limit $\lim_{\varepsilon\to 0}v_\varepsilon(x)$ exists uniformly in $x$;
define $V^L(x):=\lim_{\varepsilon\to 0}v_\varepsilon(x)$.  By
Lemma~\ref{lem:veps-cont}, each $v_\varepsilon\in C(\overline\Omega)$, and
the uniform convergence gives $V^L\in C(\overline\Omega)$.  Hence
$W^L(x,s)=V^L(x)+s$ for all $(x,s)\in\overline{D_L}$.
\end{proof}
\subsection{The higher-dimensional convergence theorem}
\label{subsec:standard-conv}

We record the standard convergence theorem for $p$-harmonious functions in
the form needed here.

\begin{theorem}[Standard convergence of $p$-harmonious functions {\cite{MPR12}}]
\label{thm:standard-conv}
Let $D\subset\R^m$ be a bounded domain satisfying the exterior sphere
condition, and let $G\in C(\partial D)$.  For each sufficiently small
$\varepsilon>0$, let $G_\varepsilon$ be a continuous extension of $G$ to the
$\varepsilon$-boundary strip of $D$, and let $U_\varepsilon$ be the unique
$p$-harmonious function in $D$ with strip boundary values $G_\varepsilon$.
Then $U_\varepsilon\to U$ uniformly on $\overline{D}$ as $\varepsilon\to 0$,
where $U$ is the unique viscosity solution of
\[
  \begin{cases}
    \DeltapN U=0 & \text{in }D,\\
    U=G & \text{on }\partial D.
  \end{cases}
\]
\end{theorem}

\subsection{Convergence by lifting}
\label{subsec:conv-lifting}

\begin{theorem}[Convergence by lifting]
\label{thm:conv}
Let $\Omega\subset\R^n$ be a bounded domain satisfying the exterior sphere
condition, and let $F\in C(\partial\Omega)$.  For each sufficiently small
$\varepsilon>0$, let $F_\varepsilon\in C(\Geps)$ be a continuous extension of
$F$, and let $v_\varepsilon$ be the unique solution of the projected dynamic
programming principle with strip boundary values $F_\varepsilon$.  Then
$v_\varepsilon\to v$ locally uniformly in $\Omega$ as $\varepsilon\to 0$,
where $v$ is the unique viscosity solution of
\begin{equation}\label{eq:limit-pde}
  \begin{cases}
    \divgg\!\bigl((1+|Dv|^2)^{p/2-1}Dv\bigr)=0 & \text{in }\Omega,\\
    v=F & \text{on }\partial\Omega.
  \end{cases}
\end{equation}
\end{theorem}

\begin{proof}
Fix $L>0$.

\medskip
\noindent\textit{Step~1.}
By Proposition~\ref{prop:lift-harmonious}, the function
$w_\varepsilon(x,s):=v_\varepsilon(x)+s$ is the unique standard
$(n+1)$-dimensional $p$-harmonious function on $D_L:=\Omega\times(-L,L)$
with continuous strip boundary values $G_\varepsilon^L(x,s)=v_\varepsilon(x)+s$.
By Lemma~\ref{lem:cylinder-bc}, $G_\varepsilon^L\to G^L:=F(x)+s$ uniformly
on $\partial D_L$.

\medskip
\noindent\textit{Step~2.}
Applying Theorem~\ref{thm:standard-conv} with $m=n+1$ and domain $D_L$ yields
$W^L\in C(\overline{D_L})$ such that $w_\varepsilon\to W^L$ uniformly on
$\overline{D_L}$, where $W^L$ is the unique viscosity solution of
\[
  \begin{cases}
    \DeltapN W^L=0 & \text{in }D_L,\\
    W^L=F(x)+s & \text{on }\partial D_L.
  \end{cases}
\]

\medskip
\noindent\textit{Step~3.}
By Corollary~\ref{cor:s-linear}, $W^L(x,s)=V^L(x)+s$ for some
$V^L\in C(\overline\Omega)$.

\medskip
\noindent\textit{Step~4.}
By Proposition~\ref{prop:visc-lift}, $V^L$ is a viscosity solution of
$\divgg\bigl((1+|DV^L|^2)^{p/2-1}DV^L\bigr)=0$ in $\Omega$.  Evaluating the
boundary condition $W^L=F(x)+s$ on $\partial D_L$ at $s=0$ gives $V^L=F$ on
$\partial\Omega$.

\medskip
\noindent\textit{Step~5.}
The regularized $p$-Laplace operator is uniformly elliptic for $p\ge 2$, so
the Dirichlet problem \eqref{eq:limit-pde} has at most one viscosity solution
\cite{CIL92}.  Since $V^L$ solves \eqref{eq:limit-pde} for every $L>0$, it
is independent of $L$; denote the common function by $v$.  Evaluating the
uniform convergence $w_\varepsilon\to W^L$ at $s=0$ gives
\[
  v_\varepsilon(x)=w_\varepsilon(x,0)\to W^L(x,0)=V^L(x)=v(x)
\]
uniformly on $\overline\Omega$ for each fixed $L$.  Since $L$ is arbitrary,
$v_\varepsilon\to v$ locally uniformly in $\Omega$.
\end{proof}

\begin{remark}
The argument is a complete reduction.  No convergence mechanism specific to
the projected model is required.  The three pillars are
Proposition~\ref{prop:lift-harmonious}, which derives the identification of
$w_\varepsilon$ with the standard $p$-harmonious function on $D_L$ from the
projected DPP alone; Corollary~\ref{cor:s-linear}, which recovers the
$s$-linearity of the limit by a translation-invariance and uniqueness argument;
and Proposition~\ref{prop:visc-lift}, which translates the higher-dimensional
convergence result back into a statement about the regularized $p$-Laplacian
in $\Omega\subset\R^n$.
\end{remark}

\bibliographystyle{amsplain}
\bibliography{ref}

\end{document}